\documentclass[10pt]{amsart}
\usepackage{amssymb}

\newtheorem{thm}{Theorem}[section]
\newtheorem{lemma}[thm]{Lemma}
\newtheorem{proposition}[thm]{Proposition}

\newtheorem{definition}[thm]{Definition}

\newtheorem{notation}[thm]{Notation}

\newcommand{\p}{\mathbb{P}}
\newcommand{\q}{\mathbb{Q}}

\newcommand{\cf}{\mathrm{cf}}
\newcommand{\cof}{\mathrm{cof}}
\newcommand{\dom}{\mathrm{dom}}

\begin{document}

\title{Forcing with Adequate Sets of Models as Side Conditions}

\author{John Krueger}

\address{Department of Mathematics \\ 
University of North Texas \\
1155 Union Circle \#311430 \\
Denton, TX 76203}

\email{jkrueger@unt.edu}

\begin{abstract}
We present a general framework for forcing on $\omega_2$ 
with finite conditions using countable models as side conditions. 
This framework is based on a method of comparing 
countable models as being membership 
related up to a large initial segment. 
We give several examples of this type of forcing, 
including adding a function on $\omega_2$, 
adding a nonreflecting stationary subset of $\omega_2 \cap \cof(\omega)$, 
and adding an $\omega_1$-Kurepa tree.
\end{abstract}

\maketitle

The method of forcing with countable models as side conditions was introduced 
by {Todor\v cevi\' c} (\cite{todor}). 
The original method is useful for forcing with finite conditions to add a generic 
object of size $\omega_1$. 
The preservation of $\omega_1$ is achieved by including finitely many 
countable elementary 
substructures as a part of a forcing condition. 
The models which appear in a condition are related by membership. 
So a condition in such a forcing poset includes a finite 
approximation of the object to be added, together with a finite 
$\in$-increasing chain of models, with some relationship specified 
between the finite fragment and the models.

Friedman (\cite{friedman}) and Mitchell (\cite{mitchell1}, \cite{mitchell2}) 
independently lifted this method up to $\omega_2$ by showing how to 
add a club subset of $\omega_2$ with finite conditions. 
In the process of going from $\omega_1$ to $\omega_2$, they gave up the 
requirement that models appearing in a forcing condition are membership related,
replacing it with a more complicated relationship between the models. 
Later Neeman (\cite{neeman}) developed 
a general approach to the subject of forcing with finite conditions 
on $\omega_2$. 
A major feature of Neeman's approach is that a condition 
in his type of forcing poset 
includes a finite $\in$-increasing chain of models, similar to 
{Todor\v cevi\' c}'s original idea, but he includes both countable and uncountable 
models in his conditions, rather than just countable models. 
Other recent papers in which side conditions are used to add objects of size 
$\omega_2$ include \cite{cummings}, 
\cite{mirna}, \cite{kos}, and \cite{boban}.

In this paper we present a general framework for forcing a generic 
object on $\omega_2$ 
with finite conditions, using countable models as side conditions. 
This framework is based on a method for comparing 
elementary substructures 
which, while not as simple as comparing by membership, 
is still natural. 
Namely, the countable models appearing in a condition will be membership 
comparable up to a large initial segment. 
The largeness of the initial segment is measured by the fact that above 
the point of comparison, the models have only a finite amount of disjoint overlap. 
We give several examples of this kind of forcing poset, including adding a 
generic function on $\omega_2$, 
adding a nonreflecting stationary subset of $\omega_2 \cap \cof(\omega)$, 
and adding an $\omega_1$-Kurepa tree. 
Since these three kinds of objects can be forced using classical methods, 
the purpose of these examples is to illustrate the method, rather than proving 
new consistency results.

\bigskip
\bigskip

This is the first in a series of papers which develop the adequate 
set approach to forcing with side conditions on $\omega_2$ 
(\cite{jk22}, \cite{jk23}, \cite{jk24}, \cite{jk25}). 
While many of the arguments appearing here could, with some work, 
be subsumed 
in the previous frameworks of Friedman, Mitchell, and Neeman, this 
paper is important for presenting the basic ideas of adequate sets 
in a way which provides a foundation for further developments.

The most important idea introduced in the paper is the parameter 
$\beta_{M,N}$, which is called 
the \emph{comparison point} of models $M$ and $N$. 
The definition of this parameter is new and does not appear explicitly 
in previous work of other authors on the subject. 
The comparision point $\beta_{M,N}$ is the basic idea behind our method for 
comparing models.

Sections 1--4 develop our framework for forcing with adequate sets 
as side conditions. 
The main goal is to develop machinery for amalgamating conditions 
over elementary substructures, which is used to preserve cardinals. 
The arguments we give for amalgamation have substantial overlap with the 
arguments for cardinal preservation of Friedman \cite{friedman} and Mitchell \cite{mitchell2}.

Sections 5--7 provide three examples of forcing posets defined with adequate sets 
as side conditions. The most important of these are adding a nonreflecting 
stationary subset of $\omega_2$ and adding an $\omega_1$-Kurepa tree. 
These applications have not appeared previously in the literature on 
forcing with finite conditions.

Our framework can be considered as 
an alternative general approach to forcing with 
finite conditions to that presented by Neeman \cite{neeman}. 
There are some equivalences between the approaches at the basic level. 
The countable models appearing in a Neeman style side condition constitute 
an adequate set, and an adequate set can be enlarged in some sense 
to a Neeman side condition. 
However, subsequent directions and generalizations 
of the theory of adequate sets, such as those in 
\cite{jk23} and \cite{jk25}, are 
incomparable with the method presented in \cite{neeman}. 
For example, forcing with adequate sets of models on $H(\lambda)$, where 
$\lambda > \omega_2$, preserves cardinals larger than 
$\omega_2$, whereas adding a 
Neeman sequence of models in $H(\lambda)$ collapses $H(\lambda)$ 
to have size $\omega_2$. 
Also coherent adequate set forcing preserves 
\textsf{CH} (\cite{jk25}), 
whereas posets defined 
in the framework of \cite{neeman} will always force that $2^\omega > \omega_1$.

I would like to thank Thomas Gilton for reading an earlier 
version of the paper and making comments and suggestions.

\section{Background Assumptions and Notation}

We make two background assumptions and fix notation for the remainder 
of the paper.

\bigskip

\noindent \textbf{Assumption 1}: $2^{\omega_1} = \omega_2$.

\bigskip

So $H(\omega_2)$ has size $\omega_2$. 

\begin{notation}
Fix a bijection $\pi : \omega_2 \to H(\omega_2)$.
\end{notation}

The importance of assumption 1 is that it implies that countable elementary 
substructures of $(H(\omega_2),\in,\pi)$ are determined by their set of ordinals. 
This allows us to use countable sets of ordinals as side conditions, instead of 
countable elementary substructures. 
An important consequence is that the forcing posets defined in this paper have 
size $\omega_2$, and hence preserve cardinals greater than $\omega_2$.

\bigskip

\noindent \textbf{Assumption 2}: There exists a stationary set 
$\mathcal Y \subseteq P_{\omega_1}(\omega_2)$ such that for all 
$\beta < \omega_2$, the set $\{ a \cap \beta : a \in \mathcal Y \}$ has size 
at most $\omega_1$.

\bigskip

A set $\mathcal Y$ as described in assumption 2 is called \emph{thin}. 
Friedman \cite{friedman} introduced the use of thin stationary sets in the 
context of forcing with models as side conditions when he used such a set 
to construct a forcing poset with finite conditions 
for adding a club to a fat stationary subset of $\omega_2$. 
Krueger \cite{krueger} proved that the existence of 
a thin stationary set does not follow from 
\textsf{ZFC}; for example, it is false 
under Martin's Maximum.
On the other hand, if \textsf{CH} holds, 
then the set $P_{\omega_1}(\omega_2)$ itself is 
thin and stationary.

Note that if $\mathcal Y$ is thin and stationary, then so is the set 
$\{ a \cap \beta : a \in \mathcal Y, \ \beta < \omega_2 \}$. 
Hence without loss of generality 
we will assume that $\mathcal Y$ is closed under initial segments. 
So for all $\beta < \omega_2$, 
$\{ a \cap \beta : a \in \mathcal Y \} = \mathcal Y \cap P(\beta)$.

\begin{notation}
Let $\mathcal A$ denote the structure $(H(\omega_2),\in,\pi,\mathcal Y)$.
\end{notation}

Since $\pi : \omega_2 \to H(\omega_2)$ is a bijection, 
if $N \prec \mathcal A$ then $N = \pi[N \cap \omega_2]$. 
Note that $\pi$ induces a definable well-ordering, and hence definable 
Skolem functions, for $\mathcal A$. 
For a set $a \subseteq H(\omega_2)$, let 
$Sk(a)$ denote the closure of $a$ under some fixed set of 
definable Skolem functions for $\mathcal A$.

\begin{lemma}
For $a \subseteq \omega_2$, $Sk(a) \cap \omega_2 = a$ iff 
$Sk(a) = \pi[a]$.
\end{lemma}

\begin{proof}
As just observed, 
$Sk(a) = \pi[Sk(a) \cap \omega_2]$. 
So if $Sk(a) \cap \omega_2 = a$, then $Sk(a) = \pi[a]$. 
Conversely, if $Sk(a) = \pi[a]$, then 
$\pi[a] = Sk(a) = \pi[Sk(a) \cap \omega_2]$. 
Since $\pi$ is one-to-one, 
the equation $\pi[a] = \pi[Sk(a) \cap \omega_2]$ 
implies that $a = Sk(a) \cap \omega_2$.
\end{proof}

\begin{lemma}
Suppose $a, b \subseteq \omega_2$, $Sk(a) \cap \omega_2 = a$, and 
$Sk(b) \cap \omega_2 = b$. 
Then $Sk(a) \cap Sk(b) = Sk(a \cap b)$.
\end{lemma}

\begin{proof}
By the previous lemma, $Sk(a) \cap Sk(b) = \pi[a] \cap \pi[b]$, which is 
equal to $\pi[a \cap b]$ since $\pi$ is injective. 
So it is enough to show that $\pi[a \cap b] = Sk(a \cap b)$. 
For this it suffices to show that $Sk(a \cap b) \cap \omega_2 = a \cap b$ 
by the previous lemma. 
Clearly $a \cap b \subseteq Sk(a \cap b) \cap \omega_2$. 
Conversely, $Sk(a \cap b) \cap \omega_2 \subseteq 
(Sk(a) \cap Sk(b)) \cap \omega_2 
= (Sk(a) \cap \omega_2) \cap (Sk(b) \cap \omega_2) = 
a \cap b$.
\end{proof}

\begin{notation}
Let $C$ denote the set of $\beta < \omega_2$ such that 
$Sk(\beta) \cap \omega_2 = \beta$.
\end{notation}

Clearly $C$ is a club.

\begin{notation}
Let $\Lambda$ denote the set of $\beta$ in $\omega_2 \cap \cof(\omega_1)$ 
such that $\beta$ is a limit point of $C$.
\end{notation}

Now we define the set $\mathcal X$ of models which will be used 
in our forcing posets.

\begin{notation}
Let $\mathcal X$ denote the set of $M \in \mathcal Y$ such that 
$Sk(M) \cap \omega_2 = M$ and for all $\gamma \in M$, 
$\sup(C \cap \gamma) \in M$.
\end{notation}

Note that $\mathcal X$ is stationary. 
If $M \in \mathcal X$, then by Lemma 1.3, $Sk(M) = \pi[M]$. 
We will sometimes refer to elements $M$ of $\mathcal X$ as \emph{models}, 
although when we do so we are informally identifying $M$ with $Sk(M)$. 
The assumption that $M$ is closed under the function which maps 
$\gamma$ to $\sup(C \cap \gamma)$ is used in Lemma 2.11, which in 
turn is used to prove Proposition 2.12.

\begin{lemma}
Let $M$ and $N$ be in $\mathcal X$, and suppose that $M \in Sk(N)$. 
Then $Sk(M) \in Sk(N)$.
\end{lemma}

\begin{proof}
Recall that $Sk(N) \prec \mathcal A = (H(\omega_2),\in,\pi,\mathcal Y)$. 
Since $M \in \mathcal X$, $Sk(M) = \pi[M]$. 
But $\pi[M]$ is definable in $\mathcal A$ from $M$ 
as the unique set $z$ such that for all $x \in M$, 
$\pi(x) \in z$, and for all $y \in z$, there is $x \in M$ such that 
$\pi(x) = y$. 
Hence $Sk(M) = \pi[M] \in Sk(N)$.
\end{proof}

\begin{lemma}
Let $M$ and $N$ be in $\mathcal X$, and suppose 
that $M \in Sk(N)$. 
Then every initial segment of $M$ is in $Sk(N)$.
\end{lemma}

\begin{proof}
Since $M \in Sk(N)$ and $M$ is countable, 
$M \subseteq Sk(N)$. 
Let $K$ be a proper initial segment of $M$. 
Let $\gamma = \min(M \setminus K)$. 
Then $K = M \cap \gamma$. 
Since $M$ and $\gamma$ are in $Sk(N)$, it follows that 
$M \cap \gamma = K$ is in $Sk(N)$.
\end{proof}

Next we relate elements of $\mathcal X$ with ordinals in $\Lambda$.
Note that by Lemma 1.4, if $M \in \mathcal X$ and $\beta \in C$, then 
$Sk(M) \cap Sk(\beta) = Sk(M \cap \beta)$. 
The next lemma says that if we cut off a set in $\mathcal X$ at an ordinal 
in $\Lambda$, then the resulting set is in $\mathcal X$.

\begin{lemma}
If $M \in \mathcal X$ and $\beta \in C$, then 
$M \cap \beta \in \mathcal X$. 
In particular, if $M \in \mathcal X$ and $\beta \in \Lambda$, then 
$M \cap \beta \in \mathcal X$.
\end{lemma}

\begin{proof}
The set $M \cap \beta$ is in $\mathcal Y$ since 
$\mathcal Y$ is closed under initial segments. 
Also $Sk(M \cap \beta) = Sk(M) \cap Sk(\beta)$. 
So $Sk(M \cap \beta) \cap \omega_2 = 
(Sk(M) \cap Sk(\beta)) \cap \omega_2 = 
(Sk(M) \cap \omega_2) \cap (Sk(\beta) \cap \omega_2) = 
M \cap \beta$. 

Now let $\gamma \in M \cap \beta$. 
Then $\sup(C \cap \gamma) \in M$ since $M \in \mathcal X$. 
But $\gamma < \beta$ implies $\sup(C \cap \gamma) \le \gamma < \beta$. 
So $\sup(C \cap \gamma) \in M \cap \beta$.
\end{proof}

The next result describes how we will use the assumption of the thinness 
of $\mathcal Y$.

\begin{proposition}
If $\beta \in \omega_2 \cap \cof(\omega_1)$, 
then $\mathcal Y \cap P(\beta) \subseteq Sk(\beta)$. 
In particular, if $M \in \mathcal X$ and $\beta \in \Lambda$, 
then $M \cap \beta \in Sk(\beta)$.
\end{proposition}

\begin{proof}
Since $\beta$ has cofinality $\omega_1$, 
it suffices to show that for all $\gamma < \beta$, 
$\mathcal Y \cap P(\gamma) \subseteq Sk(\beta)$. 
So fix $\gamma < \beta$. 
Then $\mathcal Y \cap P(\gamma) = \{ a \cap \gamma : a \in \mathcal Y \}$ 
has size at most 
$\omega_1$ by the thinness of $\mathcal Y$. 
In particular, $\mathcal Y \cap P(\gamma)$ is in $H(\omega_2)$. 
Note that $\mathcal Y \cap P(\gamma)$ 
is definable in $\mathcal A$ from $\gamma$. 
Hence $\mathcal Y \cap P(\gamma) \in Sk(\beta)$.

Again by elementarity, there is a surjection 
$g : \omega_1 \to \mathcal Y \cap P(\gamma)$ in $Sk(\beta)$. 
Since $\omega_1 \subseteq Sk(\beta)$, it follows that 
$\mathcal Y \cap P(\gamma) = g[\omega_1] \subseteq Sk(\beta)$. 
This completes the proof that 
$\mathcal Y \cap P(\beta) \subseteq Sk(\beta)$.

Now if $M \in \mathcal X$ and $\beta \in \Lambda$, then by Lemma 1.10, 
$M \cap \beta$ is in $\mathcal X \cap P(\beta)$. 
But $\mathcal X \cap P(\beta) \subseteq \mathcal Y \cap P(\beta) \subseteq 
Sk(\beta)$, so $M \cap \beta \in Sk(\beta)$.
\end{proof}

\section{Comparison Points and Remainders}

We introduce the idea of the comparison point $\beta_{M,N}$ of models 
$M, N \in \mathcal X$. 
One of the main consequences of the definition is that $M$ and $N$ will not share 
any common elements or limit points past their comparison point. 
When we use countable models as side conditions in our forcing posets, 
we will require that any two models appearing in a condition are 
membership related below their comparison point.

The definition of $\beta_{M,N}$ is made relative to a particular stationary 
subset of $\Lambda$.\footnote{For the applications in the current paper, the 
special case $\Gamma = \Lambda$ will suffice. 
In order to increase the flexibility of the method to future applications, we consider 
the more general case of a stationary subset $\Gamma$ of $\Lambda$.}

\begin{notation}
Fix for the remainder of the paper a stationary set 
$\Gamma \subseteq \Lambda$.
\end{notation}

\begin{definition}
For a set $M \in \mathcal X$, define $\Gamma_M$ as the set 
of $\beta \in \Gamma$ 
such that 
$$
\beta = \min( \Gamma \setminus (\sup(M \cap \beta)) ).
$$
\end{definition}

In other words, $\beta \in \Gamma_M$ if $\beta \in \Gamma$ and 
$$
\Gamma \cap [\sup(M \cap \beta),\beta) = \emptyset.
$$
If $\beta \in \Gamma_M$, then $\beta$ is the least element of $\Gamma$ 
strictly larger than $\sup(M \cap \beta)$. 

The set $\Gamma_M$ is countable. 
The first element of $\Gamma$ is in $\Gamma_M$. 
To produce other elements of $\Gamma_M$, if you 
take any ordinal $\gamma \le \omega_2$ 
and let $\beta := \min(\Gamma \setminus (\sup(M \cap \gamma)))$, 
then $\beta \in \Gamma_M$.

\begin{lemma}
If $M \subseteq N$ are in $\mathcal X$, then 
$\Gamma_M \subseteq \Gamma_N$.
\end{lemma}

\begin{proof}
Let $\gamma \in \Gamma_M$. 
Then by definition, 
$\gamma = \min(\Gamma \setminus (\sup(M \cap \gamma)))$. 
Since $M \subseteq N$, $\sup(M \cap \gamma) \le \sup(N \cap \gamma) < \gamma$. 
Hence $\gamma = \min(\Gamma \setminus (\sup(N \cap \gamma)))$.
\end{proof}

Note that if $\beta < \gamma$ are in $\Gamma_M$, then 
$M \cap [\beta,\gamma) \ne \emptyset$. 
For $M \cap \gamma$ cannot be a subset of $\beta$, since otherwise 
$\Gamma \cap [\sup(M \cap \gamma),\gamma)$ contains $\beta$ 
and so is nonempty.

\begin{lemma}
Let $M$ and $N$ be in $\mathcal X$. 
Then $\Gamma_M \cap \Gamma_N$ has a largest element.
\end{lemma}

\begin{proof}
The set 
$\Gamma_M \cap \Gamma_N$ is nonempty because it contains the least 
element of $\Gamma$. 
Suppose for a contradiction that $\Gamma_M \cap \Gamma_N$ has no largest element, 
and let $\gamma = \sup(\Gamma_M \cap \Gamma_N)$. 
Then $\gamma$ is a limit point of the countable set 
$\Gamma_M \cap \Gamma_N$, and therefore $\gamma$ has 
cofinality $\omega$.

Observe that if $\beta_0 < \beta_1$ are in $\Gamma_M \cap \Gamma_N$, 
then as noted before the lemma, both $M \cap [\beta_0,\beta_1)$ and 
$N \cap [\beta_0,\beta_1)$ are nonempty. 
Thus $\gamma$ is a limit point of both $M$ and $N$. 
Let $\beta$ be the minimal element of $\Gamma$ greater than or 
equal to $\gamma$. 
Since $\gamma$ has cofinality $\omega$, $\gamma < \beta$. 
Now as $\gamma$ is a limit point of both $M$ and $N$, it follows that 
$$
\gamma \le \sup(M \cap \beta), \ \gamma \le \sup(N \cap \beta),
$$
and by the choice of $\beta$, 
$\Gamma \cap [\gamma,\beta)$ is empty. 
Therefore 
$$
\Gamma \cap [\sup(M \cap \beta),\beta) = \emptyset, \ 
\Gamma \cap [\sup(N \cap \beta),\beta) = \emptyset,
$$
which implies that $\beta \in \Gamma_M \cap \Gamma_N$. 
But this contradicts 
that $\beta > \gamma$ and 
$\gamma = \sup(\Gamma_M \cap \Gamma_N)$.
\end{proof}

We now introduce the comparison point $\beta_{M,N}$ 
of models $M, N \in \mathcal X$.

\begin{notation}
For $M$ and $N$ in $\mathcal X$, let 
$\beta_{M,N}$ denote the largest ordinal in $\Gamma_M \cap \Gamma_N$.
\end{notation}

One of the most important properties 
of the comparison point of two models is that the models 
have no common elements or limit points above it.

\begin{proposition}
Let $M$ and $N$ be in $\mathcal X$. 
Let $M' := M \cup \lim(M)$ and $N' := N \cup \lim(N)$. 
Then $M' \cap N' \subseteq \beta_{M,N}$.
\end{proposition}

\begin{proof}
Suppose that $\gamma$ is in $M' \cap N'$. 
We will show that 
$\gamma < \beta_{M,N}$. 
Let $\beta$ be the least element of $\Gamma$ which is strictly greater 
than $\gamma$. 
Since $\gamma \in M'$ and $\gamma < \beta$, we have that 
$$
\gamma =  \sup(M \cap (\gamma+1)) \le \sup(M \cap \beta).
$$
Similarly, 
$$
\gamma =  \sup(N \cap (\gamma+1)) \le \sup(N \cap \beta).
$$
By the choice of $\beta$, 
$\Gamma \cap (\gamma,\beta) = \emptyset$, 
and $\sup(M \cap \beta)$ and $\sup(N \cap \beta)$ are of countable 
cofinality and hence are not in $\Gamma$. 
Therefore 
$$
\Gamma \cap [\sup(M \cap \beta),\beta) = \emptyset
$$
and 
$$
\Gamma \cap [\sup(N \cap \beta),\beta) = \emptyset.
$$
So $\beta \in \Gamma_M \cap \Gamma_N$, which implies that 
$\beta \le \beta_{M,N}$ by the maximality of $\beta_{M,N}$. 
Since $\gamma < \beta$, this proves that $\gamma < \beta_{M,N}$.
\end{proof}

The forcing posets we define later in the paper 
will contain countable models as side conditions 
which are membership 
related below their comparison point. 
Sets of models which satisfy this property will be said to be adequate.

\begin{definition}
Let $A$ be a subset of $\mathcal X$. 
We say that $A$ is \emph{adequate} if for all $M, N \in A$, 
either $M \cap \beta_{M,N} = N \cap \beta_{M,N}$, 
$M \cap \beta_{M,N} \in Sk(N)$, or $N \cap \beta_{M,N} \in Sk(M)$.
\end{definition}

Note that if $M \cap \beta_{M,N} \in Sk(N)$, then 
$M \cap \beta_{M,N} \subseteq N$ and 
$\sup(M \cap \beta_{M,N}) \in N$. 
Also by Lemma 1.8, $Sk(M \cap \beta_{M,N}) \in Sk(N)$, and by Lemma 1.9, 
every initial segment of $M \cap \beta_{M,N}$ is in $Sk(N)$.

Suppose that $\{ M, N \}$ is adequate. 
Let us show that the way in which $M$ and $N$ compare is determined by their 
intersections with $\omega_1$. 
We claim that 
$$
M \cap \beta_{M,N} \in Sk(N) \ \textrm{iff} \ M \cap \omega_1 < N \cap \omega_1.
$$
Recall that $\omega_1 \le \beta_{M,N}$ and $\omega_1 \in N$. 
In the forward direction, suppose that $M \cap \beta_{M,N} \in Sk(N)$. 
Since $\omega_1 \le \beta_{M,N}$, we have that 
$M \cap \omega_1 = (M \cap \beta_{M,N}) \cap \omega_1$. 
As $M \cap \beta_{M,N} \in Sk(N)$, by elementarity 
$$
M \cap \omega_1 = (M \cap \beta_{M,N}) \cap \omega_1 \in Sk(N) \cap \omega_1 = 
N \cap \omega_1.
$$
Conversely, assume that $M \cap \omega_1 < N \cap \omega_1$. 
By the forward direction just proven, if $N \cap \beta_{M,N} \in Sk(M)$, 
then $N \cap \omega_1 < M \cap \omega_1$, which contradicts that 
$M \cap \omega_1 < N \cap \omega_1$. 
On the other hand, if $M \cap \beta_{M,N} = N \cap \beta_{M,N}$, then 
$$
M \cap \omega_1 = (M \cap \beta_{M,N}) \cap \omega_1 = 
(N \cap \beta_{M,N}) \cap \omega_1 = N \cap \omega_1,
$$
which again contradicts that $M \cap \omega_1 < N \cap \omega_1$. 
Hence the only possible way in which $M$ and $N$ could compare is that 
$M \cap \beta_{M,N} \in Sk(N)$, which proves the claim. 

It easily follows from this claim that 
$$
M \cap \beta_{M,N} = N \cap \beta_{M,N} \ \textrm{iff} \ 
M \cap \omega_1 = N \cap \omega_1.
$$
For the failure of the first statement implies that $M \cap \omega_1$ and 
$N \cap \omega_1$ are not equal by the claim, and conversely if these ordinals are not equal then the claim implies that either $M \cap \beta_{M,N} \in Sk(M)$ or 
$N \cap \beta_{M,N} \in Sk(N)$, depending on which ordinal is larger.

\bigskip

If $A$ is an adequate set and $M \in A$, we say that $M$ is 
\emph{$\in$-minimal} in $A$ if for all $N \in A$, 
$M \cap \omega_1 \le N \cap \omega_1$. 
Note that there always exists an $\in$-minimal model in $A$, if $A$ is nonempty. 
Also by the previous two paragraphs, $M \in A$ is minimal iff for all 
$N$ in $A$, $M \cap \beta_{M,N}$ 
is either equal to $N \cap \beta_{M,N}$ or in $Sk(N)$.

\bigskip

Now we introduce the idea of the remainder set, which describes the 
disjoint overlap of models above their comparison point.

\begin{definition}
Let $\{ M, N \}$ be adequate. 
Define the \emph{remainder set of $N$ over $M$}, denoted by 
$R_M(N)$, as the set of $\beta$ satisfying either:
\begin{enumerate}
\item there is $\gamma \ge \beta_{M,N}$ in $M$ such that 
$\beta = \min(N \setminus \gamma)$, or
\item $N \cap \beta_{M,N}$ is either equal to $M \cap \beta_{M,N}$ or 
is in $Sk(M)$, and $\beta = \min(N \setminus \beta_{M,N})$.
\end{enumerate}
\end{definition}

Note that we do not explicitly require the ordinal 
$\min(N \setminus \beta_{M,N})$ to be in $R_M(N)$ in the case that 
$M \cap \beta_{M,N} \in Sk(N)$.

\begin{proposition}
Let $\{ M, N \}$ be adequate. 
Then $R_M(N)$ is finite.
\end{proposition}

\begin{proof}
Suppose not, and let $\langle \beta_{n} : n < \omega \rangle$ 
be a strictly increasing sequence of ordinals in $R_M(N)$. 
Let $\xi = \sup_n \beta_n$. 
Then $\xi$ is a limit point of $N$. 
By the definition of $R_M(N)$, 
for each $n$ we can fix $\gamma_n \in M \cap (\beta_n,\beta_{n+1})$. 
Then $\xi = \sup_n \gamma_n$. 
So $\xi$ is a common limit point of $M$ and $N$ which is above 
$\beta_{M,N}$, which contradicts Proposition 2.6.
\end{proof}

\begin{lemma}
Let $\{ M, N \}$ be adequate. 
Let $\beta \in R_M(N)$, and suppose that $\beta$ is not equal to 
$\min(N \setminus \beta_{M,N})$. 
Then there is $\gamma \in R_N(M)$ such that 
$\beta = \min(N \setminus \gamma)$.
\end{lemma}

\begin{proof}
Suppose that $\beta \in R_M(N)$ and 
$\beta$ is not equal to $\min(N \setminus \beta_{M,N})$. 
Then by the definition of $R_M(N)$, we can fix 
$\gamma^* \in M \setminus \beta_{M,N}$ such that 
$\beta = \min(N \setminus \gamma^*)$. 
Since $\beta$ is not equal to $\min(N \setminus \beta_{M,N})$, 
fix $\beta^* \in N \setminus \beta_{M,N}$ which is below $\beta$. 
Then 
$$
\beta_{M,N} \le \beta^* < \gamma^* < \beta.
$$

We claim that there exists some $\xi$ in $R_N(M)$ with 
$\beta^* < \xi \le \gamma^*$. 
Namely, let $\xi := \min(M \setminus \beta^*)$. 
Now let $\gamma$ be the largest such $\xi$, which 
is possible since $R_N(M)$ is finite. 
So 
$$
\beta_{M,N} \le \beta^* < \gamma \le \gamma^* < \beta.
$$
Clearly there is no ordinal in $N$ between $\gamma$ and $\gamma^*$, since otherwise 
the least member of $M$ above it would be in $R_N(M)$, 
contradicting the maximality of $\gamma$. 
Since $\beta$ is the least member of $N$ above $\gamma^*$, and 
$N \cap [\gamma,\gamma^*] = \emptyset$, 
it follows that $\beta = \min(N \setminus \gamma)$.
\end{proof}

We would now like to show that $R_M(N)$ is always a subset of $\Gamma$ in 
the case when $\Gamma = \Lambda$. 
This follows from Proposition 2.12, which is proved using Lemma 2.11.

\begin{lemma}
Let $M$ be in $\mathcal X$, $\beta \in M$, and suppose that 
$$
C \cap (\sup(M \cap \beta),\beta) \ne \emptyset.
$$
Then $\beta \in \Lambda$.
\end{lemma}

\begin{proof}
Since $C \cap (\sup(M \cap \beta),\beta)$ is nonempty, obviously 
$\sup(M \cap \beta) < \beta$. 
This implies that $\beta$ has cofinality $\omega_1$. 
For if $\beta$ has countable cofinality, 
then easily by elementarity, $M \cap \beta$ is cofinal in $\beta$, 
which contradicts that $\sup(M \cap \beta) < \beta$.

By the definition of $\Lambda$, to show that $\beta$ is in $\Lambda$ 
it suffices to show that $\beta$ is a limit point of $C$.
Suppose for a contradiction that $\beta$ is not a limit point of $C$. 
Then $\sup(C \cap \beta) < \beta$. 
Since $M \in \mathcal X$, by the definition of $\mathcal X$ it follows that 
$\sup(C \cap \beta) \in M \cap \beta$. 
But by assumption, there is $\gamma \in C$ with 
$\sup(M \cap \beta) < \gamma < \beta$, which is a contradiction.
\end{proof}

\begin{proposition}
Let $\{ M, N \}$ be adequate. 
Then $R_M(N)$ and $R_N(M)$ are subsets of $\Lambda$.
\end{proposition}

\begin{proof}
We prove by induction on $\alpha$ that if $\alpha \ge \beta_{M,N}$ is 
in $R_M(N) \cup R_N(M)$, then $\alpha \in \Lambda$. 
So let $\alpha$ be given, and assume that the statement is true for all smaller ordinals. 
We handle only the case when $\alpha \in R_N(M)$, since the proof of the 
case when $\alpha \in R_M(N)$ is the same except with the roles of 
$M$ and $N$ reversed.

First, suppose that $\alpha = \min(M \setminus \beta_{M,N})$. 
If $\alpha = \beta_{M,N}$, then $\alpha \in \Lambda$ by definition. 
Otherwise 
$$
\sup(M \cap \alpha) < \beta_{M,N} < \alpha.
$$
So $\beta_{M,N} \in C \cap (\sup(M \cap \alpha),\alpha)$, which implies that 
$\alpha \in \Lambda$ by Lemma 2.11.

Secondly, suppose that 
$\alpha$ is not equal to $\min(M \setminus \beta_{M,N})$, and 
$\alpha = \min(M \setminus \gamma)$ for some 
$\gamma \in N \setminus \beta_{M,N}$. 
By Lemma 2.10, without loss of generality 
we may assume that $\gamma \in R_M(N)$. 
By the inductive hypothesis, $\gamma \in \Lambda \subseteq C$. 
Clearly 
$$
\sup(M \cap \alpha) < \gamma < \alpha.
$$
So $C \cap (\sup(M \cap \alpha),\alpha) \ne \emptyset$. 
By Lemma 2.11, $\alpha \in \Lambda$.
\end{proof}

\section{Adequate Sets of Models}

In this section we introduce methods for 
extending adequate sets of models to larger adequate sets. 
The use of these methods for preserving cardinals in forcing with models as 
side conditions will be demonstrated in the next section.

First we prove a couple of technical lemmas.

\begin{lemma}
Let $M \in \mathcal X$, $\beta \in \Gamma$, and suppose that 
$M \subseteq \beta$. 
Then $\Gamma_M \subseteq \beta+1$. 
Therefore for all $N \in \mathcal X$, $\beta_{M,N} \le \beta$.
\end{lemma}

\begin{proof}
Since $M \subseteq \beta$ and $\cf(\beta) = \omega_1$, 
$\sup(M) < \beta$. 
Let $\gamma \in \Gamma_M$ be given. 
Then $\sup(M \cap \gamma) \le \sup(M) < \beta$. 
Since $\beta \in \Gamma$ and $\gamma = 
\min(\Gamma \setminus \sup(M \cap \gamma))$, it follows that 
$\gamma \le \beta$. 
This proves that $\Gamma_M \subseteq \beta+1$. 
In particular, if $N \in \mathcal X$, 
then by definition, $\beta_{M,N} \in \Gamma_M$, so $\beta_{M,N} \le \beta$.
\end{proof}

\begin{lemma}
Let $K, M, N \in \mathcal X$, and suppose that $M \subseteq N$. 
Then $\beta_{M,K} \le \beta_{N,K}$.
\end{lemma}

\begin{proof}
Since $M \subseteq N$, $\Gamma_M \subseteq \Gamma_N$ by Lemma 2.3. 
So $\Gamma_M \cap \Gamma_K \subseteq \Gamma_N \cap \Gamma_K$. 
Hence $\beta_{M,K} = \max( \Gamma_M \cap \Gamma_K ) \le 
\max( \Gamma_N \cap \Gamma_K ) = \beta_{N,K}$.
\end{proof}

The next two results show that if you start with an adequate set $A$, 
and add to $A$ models of the form $M \cap \beta$, where $M \in A$ 
and $\beta \in \Gamma$, then the bigger set is also adequate.

\begin{lemma}
Suppose that $\{ M, N \}$ is adequate and $\beta \in \Gamma$. 
Then $\{ M \cap \beta, N \}$ is adequate.
\end{lemma}

\begin{proof}
Since $M \cap \beta \subseteq M$, 
$\beta_{M \cap \beta,N} \le \beta_{M,N}$ by Lemma 3.2. 
Also since $M \cap \beta \subseteq \beta$, 
$\beta_{M \cap \beta,N} \le \beta$ by Lemma 3.1. 

To show that $\{ M \cap \beta, N \}$ is adequate, we split into three 
cases depending on how $M$ and $N$ compare.

(1) Suppose that $M \cap \beta_{M,N} = N \cap \beta_{M,N}$. 
Since $\beta_{M \cap \beta,N} \le \beta_{M,N}$, we get that 
$$
M \cap \beta_{M \cap \beta,N} = N \cap \beta_{M \cap \beta,N}.
$$
As $\beta_{M \cap \beta,N} \le \beta$,
$$
(M \cap \beta) \cap \beta_{M \cap \beta,N} = 
M \cap \beta_{M \cap \beta,N} = N \cap \beta_{M \cap \beta,N}.
$$

(2) Suppose that $M \cap \beta_{M,N} \in Sk(N)$. 
Since $\beta_{M \cap \beta,N} \le \beta$, we have that 
$(M \cap \beta) \cap \beta_{M \cap \beta,N} = 
M \cap \beta_{M \cap \beta,N}$. 
As $\beta_{M \cap \beta,N} \le \beta_{M,N}$, it follows that 
$M \cap \beta_{M \cap \beta,N}$ is an initial segment of 
$M \cap \beta_{M,N}$. 
But $M \cap \beta_{M,N} \in Sk(N)$, so the initial segment 
$(M \cap \beta) \cap \beta_{M \cap \beta,N} = M \cap \beta_{M \cap \beta,N}$ 
is in $Sk(N)$.

(3) Suppose that $N \cap \beta_{M,N} \in Sk(M)$. 
Then $N \cap \beta_{M \cap \beta,N} \in Sk(M)$, since the inequality 
$\beta_{M \cap \beta,N} \le \beta_{M,N}$ implies that 
$N \cap \beta_{M \cap \beta,N}$ 
it is an initial segment of $N \cap \beta_{M,N}$. 
By Proposition 1.11 and the inequality $\beta_{M \cap \beta,N} \le \beta$, 
we have that 
$$
N \cap \beta_{M \cap \beta,N} \in 
Sk(\beta_{M \cap \beta,N}) \subseteq Sk(\beta).
$$
So by Lemma 1.4, 
$$
N \cap \beta_{M \cap \beta,N} \in Sk(M) \cap Sk(\beta) = 
Sk(M \cap \beta).
$$
\end{proof}

\begin{proposition}
Suppose that $A$ is adequate, $A \subseteq B \subseteq \mathcal X$, 
and for all 
$K \in B \setminus A$, there is $M \in A$ and $\beta \in \Gamma$ 
such that $K = M \cap \beta$. 
Then $B$ is adequate.
\end{proposition}

\begin{proof}
It suffices to show that for all $K, L \in B$, the set $\{ K, L \}$ is adequate. 
By Lemma 3.3 and the fact that $A$ is adequate, 
this is true if at least one of $K$ or $L$ is in $A$. 
So assume that $K$ and $L$ are both in $B \setminus A$. 
Fix $M, N \in A$ and $\beta, \gamma \in \Gamma$ such that 
$K = M \cap \beta$ and $L = N \cap \gamma$. 
Then $\{ M \cap \beta, N \}$ is adequate by Lemma 3.3. 
Hence $\{ M \cap \beta, N \cap \gamma \}$ is adequate again by Lemma 3.3.
\end{proof}

The next result says that adding to an adequate set $A$ a model whose 
Skolem hull contains the elements of $A$ results in an adequate set.

\begin{proposition}
Let $A$ be adequate, and let $N \in \mathcal X$ satisfy that 
$A \subseteq Sk(N)$. 
Then $A \cup \{ N \}$ is adequate. 
In particular, if $M$ and $N$ are in $\mathcal X$ and $M \in Sk(N)$, 
then $\{ M, N \}$ is adequate.
\end{proposition}

\begin{proof}
Let $M \in A$. 
Then $M \in Sk(N)$, which implies that $\sup(M) \in N$. 
Hence $\beta_{M,N} > \sup(M)$ by Proposition 2.6. 
Thus $M \cap \beta_{M,N} = M \in Sk(N)$.
\end{proof}

An essential part of the arguments for preserving cardinals in 
forcing with models as side conditions will be 
to amalgamate conditions over elementary substructures. 
In particular, this involves amalgamating adequate sets of models. 
Amalgamation over countable models is handled in 
Proposition 3.9, and amalgamation over models of size $\omega_1$ 
is handled in Proposition 3.11.

First we prove two technical lemmas.

\begin{lemma}
Let $M$ and $N$ be in $\mathcal X$ and let $\beta \in \Gamma$. 
If $\beta_{M,N} \le \beta$, then 
$\beta_{M,N} = \beta_{M \cap \beta,N}$.
\end{lemma}

\begin{proof}
Since $\beta_{M,N} \le \beta$, 
$$
\sup((M \cap \beta) \cap \beta_{M,N}) = \sup(M \cap \beta_{M,N}).
$$
Therefore  
$$
\min(\Gamma \setminus \sup((M \cap \beta) \cap \beta_{M,N})) = 
\min(\Gamma \setminus \sup(M \cap \beta_{M,N})) = \beta_{M,N}.
$$
By the definition of $\Gamma_{M \cap \beta}$, we have that 
$\beta_{M,N} \in \Gamma_{M \cap \beta}$. 
It follows that $\beta_{M,N}$ is the largest element of 
$\Gamma_{M \cap \beta} \cap \Gamma_N$, since it is the largest 
element of $\Gamma_M \cap \Gamma_N$ by definition, and 
$\Gamma_{M \cap \beta} \cap \Gamma_N \subseteq 
\Gamma_M \cap \Gamma_N$ by Lemma 2.3. 
So $\beta_{M,N} = \beta_{M \cap \beta,N}$.
\end{proof}

\begin{lemma}
Let $M$ and $N$ be in $\mathcal X$ and let $\beta \in \Gamma$. 
If $N \subseteq \beta$, then 
$\beta_{M,N} = \beta_{M \cap \beta,N}$.
\end{lemma}

\begin{proof}
By the previous lemma, it suffices to show that 
$\beta_{M,N} \le \beta$. 
This follows from Lemma 3.1.
\end{proof}

We are ready to handle amalgamation of adequate sets over countable 
elementary substructures.

\begin{definition}
Let $A$ be adequate and $N \in \mathcal X$. 
We say that $A$ is \emph{$N$-closed} if 
for all $M \in A$, if $M \cap \beta_{M,N} \in Sk(N)$, then 
$M \cap \beta_{M,N} \in A$.
\end{definition}

Note that if $A$ is adequate and $N \in \mathcal X$, 
then by Proposition 3.4, the set 
$$
A \cup \{ M \cap \beta_{M,N} : M \in A, \ 
M \cap \beta_{M,N} \in Sk(N) \}
$$
is adequate and $N$-closed.

Observe that a set $A$ is adequate iff for all $M$ and $N$ in $A$, 
$\{ M, N \}$ is adequate.

\begin{proposition}
Let $A$ be adequate, $N \in A$, and suppose that $A$ is $N$-closed. 
Let $B$ be adequate such that 
$$
A \cap Sk(N) \subseteq B \subseteq Sk(N).
$$
Then $A \cup B$ is adequate.
\end{proposition}

\begin{proof}
Since $A$ and $B$ are each adequate, it suffices to show that for all 
$M \in A$ and $L \in B$, the pair $\{ L, M \}$ is adequate. 
So let $M \in A$ and $L \in B$. 
As $B \subseteq Sk(N)$, we have that $L \in Sk(N)$.

In the easy case that $M \in Sk(N)$, 
we have that $M \in A \cap Sk(N) \subseteq B$. 
So $L$ and $M$ are both in $B$. 
As $B$ is adequate, we are done. 
Assume for the rest of the proof 
that $M \in A \setminus Sk(N)$.

Since $L \in Sk(N)$, it follows that 
(a) $\beta_{L,M} \le \beta_{M,N}$ by Lemma 3.2. 
So by Lemma 3.6, 
(b) $\beta_{L,M} = \beta_{L,M \cap \beta_{M,N}}$. 

As $M$ and $N$ are in $A$, the set $\{ M, N \}$ is adequate. 
We split the proof into three 
cases depending on the type of comparison which holds between 
$M$ and $N$.

\bigskip

(1) Assume that $M \cap \beta_{M,N} = N \cap \beta_{M,N}$. 
We will show that $L \cap \beta_{L,M} \in Sk(M)$. 
Since $L \in Sk(N)$, $L \cap \beta_{M,N} \in Sk(N)$, since $L \cap \beta_{M,N}$ 
is an initial segment of $L$. 
By Proposition 1.11, $L \cap \beta_{M,N} \in Sk(\beta_{M,N})$. 
So 
$$
L \cap \beta_{M,N} \in Sk(N) \cap Sk(\beta_{M,N}) = Sk(N \cap \beta_{M,N}).
$$
But since $M \cap \beta_{M,N} = N \cap \beta_{M,N}$, we have that 
$$
Sk(N \cap \beta_{M,N}) = Sk(M \cap \beta_{M,N}) \subseteq Sk(M).
$$
So $L \cap \beta_{M,N} \in Sk(M)$. 
Since $\beta_{L,M} \le \beta_{M,N}$ by (a) above, 
it follows that $L \cap \beta_{L,M} \in Sk(M)$.

\bigskip

(2) Assume that $N \cap \beta_{M,N} \in Sk(M)$. 
We will show that $L \cap \beta_{L,M} \in Sk(M)$. 
Since $L \in Sk(N)$, $L \cap \beta_{M,N} \in Sk(N)$, since $L \cap \beta_{M,N}$ is 
an initial segment of $L$. 
By Proposition 1.11, $L \cap \beta_{M,N} \in Sk(\beta_{M,N})$. 
So 
$$
L \cap \beta_{M,N} \in Sk(N) \cap Sk(\beta_{M,N}) = Sk(N \cap \beta_{M,N}).
$$
But since $N \cap \beta_{M,N} \in Sk(M)$, we have that 
$$
Sk(N \cap \beta_{M,N}) \subseteq Sk(M).
$$
Thus $L \cap \beta_{M,N} \in Sk(M)$. 
As $\beta_{L,M} \le \beta_{M,N}$ by (a) above,
$L \cap \beta_{L,M} \in Sk(M)$.

\bigskip

(3) Suppose that $M \cap \beta_{M,N} \in Sk(N)$. 
Since $A$ is $N$-closed, 
$M \cap \beta_{M,N} \in A$. 
So $M \cap \beta_{M,N} \in A \cap Sk(N) \subseteq B$. 
Hence $L$ and $M \cap \beta_{M,N}$ are both in $B$. 
As $B$ is adequate, it follows that $L$ and $M \cap \beta_{M,N}$ 
compare properly.

We claim that 
$$
(M \cap \beta_{M,N}) \cap \beta_{L,M \cap \beta_{M,N}} = M \cap \beta_{L,M}.
$$
As $\beta_{L,M} = \beta_{L,M \cap \beta_{M,N}}$ by (b) above, we have that 
$$
(M \cap \beta_{M,N}) \cap \beta_{L,M \cap \beta_{M,N}} = 
(M \cap \beta_{M,N}) \cap \beta_{L,M}.
$$
And since $\beta_{L,M} \le \beta_{M,N}$ by (a) above, 
$$
(M \cap \beta_{M,N}) \cap \beta_{L,M} = M \cap \beta_{L,M}.
$$
This proves the claim.

We consider the three possible comparisons of $L$ and $M \cap \beta_{M,N}$. 
First, suppose that 
$$
(M \cap \beta_{M,N}) \cap \beta_{L,M \cap \beta_{M,N}} \in Sk(L).
$$
Then by the claim, 
$$
M \cap \beta_{L,M} \in Sk(L),
$$
and we are done. 
Secondly, assume that 
$$
L \cap \beta_{L,M \cap \beta_{M,N}} \in Sk(M \cap \beta_{M,N}).
$$
Since $\beta_{L,M} = \beta_{L,M \cap \beta_{M,N}}$ by (b) above, 
it follows that 
$$
L \cap \beta_{L,M} \in Sk(M \cap \beta_{M,N}) \subseteq Sk(M),
$$
and hence $L \cap \beta_{L,M}\in Sk(M)$, which finishes the proof. 
Thirdly, if $$
L \cap \beta_{L,M \cap \beta_{M,N}} = (M \cap \beta_{M,N}) \cap 
\beta_{L,M \cap \beta_{M,N}},
$$
then by (b) and the claim, 
$$
L \cap \beta_{L,M} = M \cap \beta_{L,M}.
$$
\end{proof}

Next we handle amalgamation of adequate sets over elementary 
substructures of size $\omega_1$.

\begin{definition}
Let $A$ be adequate, and let $\beta \in \Gamma$. 
We say that $A$ is \emph{$\beta$-closed} if 
for all $M \in A$, $M \cap \beta \in A$.
\end{definition}

Note that if $A$ is adequate and $\beta \in \Gamma$, then by Proposition 3.4, 
the set 
$$
A \cup \{ M \cap \beta : M \in A \}
$$
is adequate and $\beta$-closed.

\begin{proposition}
Let $A$ be adequate, $\beta \in \Gamma$, and suppose 
that $A$ is $\beta$-closed. 
Let $B$ be adequate such that 
$$
A \cap P(\beta) \subseteq B \subseteq P(\beta).
$$
Then $A \cup B$ is adequate.
\end{proposition}

\begin{proof}
Consider $N \in A$ and $M \in B$, and we will show 
that $\{ M, N \}$ is adequate. 
If $N \subseteq \beta$, then $N \in A \cap P(\beta) \subseteq B$, so 
both $M$ and $N$ are in $B$. 
Since $B$ is adequate, so is $\{ M, N \}$, and we are done. 
Thus we will assume for the rest of the proof that $N \in A \setminus P(\beta)$.

Since $A$ is $\beta$-closed, 
$$
N \cap \beta \in A \cap P(\beta).
$$
As $A \cap P(\beta) \subseteq B$, $N \cap \beta \in B$. 
So both $M$ and $N \cap \beta$ are in $B$. 
Since $B$ is adequate, so is $\{ M, N \cap \beta \}$. 

Note that 
since $M \subseteq \beta$, we have that (a) 
$\beta_{M,N} = \beta_{M,N \cap \beta}$ 
by Lemma 3.7. 
By Lemma 3.1, $M \subseteq \beta$ implies that 
(b) $\beta_{M,N} \le \beta$. 

The rest of the proof will split into the three cases of how $M$ and $N \cap \beta$ 
compare.

\bigskip

(1) Suppose that 
$$
M \cap \beta_{M,N \cap \beta} \in Sk(N \cap \beta).
$$
Since $\beta_{M,N} = \beta_{M,N \cap \beta}$ by (a) above, it follows that 
$$
M \cap \beta_{M,N} \in Sk(N \cap \beta) \subseteq Sk(N).
$$
So $M \cap \beta_{M,N} \in Sk(N)$, and we are done. 

\bigskip

We make an additional observation to handle cases (2) and (3). 
Since $\beta_{M,N} = \beta_{M,N \cap \beta}$ by (a) above, and 
$\beta_{M,N} \le \beta$ by (b) above, we have that 
$$
(N \cap \beta) \cap \beta_{M,N \cap \beta} = 
(N \cap \beta) \cap \beta_{M,N} = N \cap \beta_{M,N}.
$$

\bigskip

(2) Suppose that 
$$
(N \cap \beta) \cap \beta_{M,N \cap \beta} = 
M \cap \beta_{M,N \cap \beta}.
$$
It follows that 
$$
N \cap \beta_{M,N} = 
(N \cap \beta) \cap \beta_{M,N \cap \beta} = M \cap \beta_{M,N \cap \beta} = 
M \cap \beta_{M,N},
$$
where the last equality holds by (a). 

\bigskip

(3) Suppose that 
$$
(N \cap \beta) \cap \beta_{M,N \cap \beta} \in Sk(M).
$$
Since $(N \cap \beta) \cap \beta_{M,N \cap \beta} = N \cap \beta_{M,N}$, 
we have that 
$$
N \cap \beta_{M,N} \in Sk(M).
$$
\end{proof}

\section{Forcing with Adequate Sets of Models}

We now present a simple example to illustrate how the results from the 
last section can be used to preserve cardinals in forcing with adequate 
sets of models as side conditions.

Recall the following definitions of Mitchell \cite{mitchell2}. 
Let $\q$ be a forcing poset, $q \in \q$, and $N$ a set. 
We say that $q$ is a \emph{strongly $(N,\q)$-generic condition} 
if for any set $D$ which is a dense subset of the forcing poset $N \cap \q$, 
$D$ is predense in $\q$ below $q$. 
The forcing poset $\q$ is said to be \emph{strongly proper on a stationary set} 
if for any sufficiently large regular cardinal $\theta$ with $\q \subseteq H(\theta)$, 
there are stationarily many countable $N \prec H(\theta)$ such that for 
every condition $p \in N \cap \q$, there is an extension $q \le p$ which is 
strongly $(N,\q)$-generic.

Standard proper forcing arguments show that if $\q$ is strongly proper on 
a stationary set, then $\q$ preserves $\omega_1$. 
More generally, let $\kappa$ be a regular uncountable cardinal. 
Assume that for any sufficiently large regular cardinal $\lambda \ge \kappa$ 
with $\q \subseteq H(\lambda)$, there are stationarily many 
$N$ in $P_{\kappa}(H(\lambda))$ such that $N \cap \kappa \in \kappa$ and 
every condition in $N \cap \q$ has a strongly $(N,\q)$-generic extension. 
Then $\q$ preserves the cardinal $\kappa$.

\begin{definition}
Let $\p$ be the forcing poset whose conditions are finite adequate sets. 
Let $B \le A$ if $A \subseteq B$.
\end{definition}

\begin{proposition}
The forcing poset $\p$ is strongly proper on a stationary set. 
In particular, $\p$ preserves $\omega_1$.
\end{proposition}

\begin{proof}
Fix $\theta > \omega_2$ regular. 
Let $N^*$ be a countable elementary substructure of $H(\theta)$ satisfying 
that $\p, \pi, \mathcal X \in N^*$ 
and $N := N^* \cap \omega_2 \in \mathcal X$. 
Note that since $\mathcal X$ is stationary, there are stationarily many such $N^*$ in 
$P_{\omega_1}(H(\theta))$. 
So to prove the proposition, it suffices to show that 
every condition in 
$N^* \cap \p$ has a strongly $(N^*,\p)$-generic extension.

Observe that since $\pi \in N^*$ and $\pi : \omega_2 \to H(\omega_2)$ 
is a bijection, by elementarity we have that 
$$
N^* \cap H(\omega_2) = \pi[N^* \cap \omega_2] = \pi[N] = Sk(N),
$$
where the last equality holds by Lemma 1.3 and the fact that 
$N \in \mathcal X$ implies that $Sk(N) \cap \omega_2 = N$. 
In particular, $N^* \cap \p \subseteq Sk(N)$.

Let $A \in N^* \cap \p$, and we will find an extension of $A$ 
which is strongly $(N^*,\p)$-generic. 
Define 
$$
B := A \cup \{ N \}.
$$
By Lemma 3.5, $B$ is adequate. 
So $B \in \p$, and clearly $B \le A$. 
We will show that $B$ is strongly $(N^*,\p)$-generic, which finishes the proof. 
Fix a set $E$ which is a dense subset of $N^* \cap \p$, and we will 
show that $E$ is predense below $B$.

Let $C \le B$. 
We will find a condition in $E$ which is compatible with $C$. 
To prepare for intersecting with $N^*$, we will first extend $C$. 
Define 
$$
D := C \cup \{ M \cap \beta_{M,N} : M \in C, \ 
M \cap \beta_{M,N} \in Sk(N) \}.
$$
Then $D$ is finite, adequate, and $N$-closed. 
Since $D \le C$, it suffices to find a condition in $E$ which is compatible with $D$.

Define $X := D \cap N^*$. 
Then $X$ is in $\p$. 
Since $X$ is a finite subset of $N^*$, $X \in N^*$. 
Also note that since $N^* \cap \p \subseteq Sk(N)$, 
$X = D \cap Sk(N)$. 

As $E$ is dense in $N^* \cap \p$, we can fix 
$Y \le X$ in $E$. 
Now $E \subseteq N^* \cap \p \subseteq Sk(N)$. 
So $Y \in Sk(N)$. 
Since $Y \in E$, we will be finished if we can show that $Y$ is compatible with $D$.

We apply Proposition 3.9. 
We have that $D$ is adequate, $N \in D$, and $D$ is $N$-closed. 
Moreover, $Y$ is adequate, and 
$$
D \cap Sk(N) = X \subseteq Y \subseteq Sk(N).
$$
By Proposition 3.9, it follows that $D \cup Y$ is adequate. 
Hence $D \cup Y$ is a condition below 
$D$ and $Y$, showing that $D$ and $Y$ are compatible.
\end{proof}

The preservation of $\omega_2$ involves amalgamating conditions 
over a model of size $\omega_1$. 
This argument sometimes shows that the forcing poset under consideration 
is $\omega_2$-c.c., using the next lemma.

\begin{lemma}
Let $\q$ be a forcing poset. 
Fix $\theta > \omega_2$ with $\q \in H(\theta)$. 
Suppose that there 
exists $N^* \prec H(\theta)$ of size at most $\omega_1$ with $\q \in N^*$ 
such that the empty 
condition is strongly $(N^*,\q)$-generic.\footnote{It actually suffices that the empty 
condition is $(N^*,\q)$-generic, in the sense of proper forcing, which is a weaker 
assumption. But the lemma is 
stated in the form which we will use.} 
Then $\q$ is $\omega_2$-c.c.
\end{lemma}

\begin{proof}
Suppose for a contradiction that $\q$ is not $\omega_2$-c.c. 
By elementarity, we can fix an antichain $A$ of $\q$ in $N^*$ 
such that $|A| \ge \omega_2$. 
Since $N^*$ has size at most $\omega_1$ and $A$ has size greater than $\omega_1$, 
we can fix a condition $q$ which is in $A \setminus N^*$.

Let $D$ be the dense set of conditions which are below some condition in $A$. 
Then $D \in N^*$ by elementarity. 
Again by elementarity, $N^* \cap D$ is a dense subset of the forcing 
poset $\q \cap N^*$.

Since the empty condition is strongly $(N^*,\q)$-generic, $N^* \cap D$ 
is predense in the forcing poset $\q$. 
In particular, we can find $w \in N^* \cap D$ which is compatible 
with the condition $q$. 
By the definition of $D$, there is some $u \in A$ such that $w \le u$, and since 
$w \in N^*$, by elementarity there is such a $u$ in $N^*$. 
Since $w$ is compatible with $q$, and $w \le u$, it follows that 
$u$ and $q$ are compatible. 
But $u \in N^* \cap A$ and $q \in A \setminus N^*$, hence $u \ne q$. 
So $q$ and $u$ are distinct conditions in $A$ which are compatible, contradicting 
the fact that $A$ is an antichain.
\end{proof}

We use Proposition 3.11 to prove that $\p$ preserves $\omega_2$.

\begin{proposition}
The forcing poset $\p$ is $\omega_2$-c.c.
\end{proposition}

\begin{proof}
Let $\theta > \omega_2$ be regular such that $\p \in H(\theta)$. 
Fix $N^* \prec H(\theta)$ of size $\omega_1$ such that 
$\p, \pi, \mathcal X \in N^*$ and 
$\beta^* := N^* \cap \omega_2 \in \Gamma$. 
Note that this is possible since $\Gamma$ is stationary. 
Since $\pi \in N^*$ and $\pi : \omega_2 \to H(\omega_2)$ is a bijection, 
by elementarity we have that 
$$
N^* \cap H(\omega_2) = \pi[N^* \cap \omega_2] = \pi[\beta^*] = Sk(\beta^*),
$$
where the last equality holds by Lemma 1.3 and the fact that 
$\beta^* \in \Gamma$ implies that $Sk(\beta^*) \cap \omega_2 = \beta^*$. 
In particular, $N^* \cap \p \subseteq Sk(\beta^*)$.

We will prove that the empty condition is strongly $(N^*,\p)$-generic. 
By Lemma 4.3, this implies that $\p$ is $\omega_2$-c.c., which finishes the proof. 
So fix $E$ which is a dense subset of $N^* \cap \p$, and we will show that 
$E$ is predense in $\p$.
 
Let $B \in \p$ be given. 
We will find a condition in $E$ which is compatible with $B$. 
First we extend $B$ to prepare for intersecting with $N^*$. 
Define 
$$
C := B \cup \{ M \cap \beta^* : M \in B \}.
$$
Then $C$ is finite, adequate, and $\beta^*$-closed.
Since $C \le B$, it suffices to find a condition in $E$ which is compatible with $C$.

We claim that 
$$
N^* \cap C = C \cap P(\beta^*).
$$
On the one hand, $N^* \cap C \subseteq C \cap P(\beta^*)$ since 
$N^* \cap \omega_2 = \beta^*$. 
Conversely, by Proposition 1.11, 
$$
C \cap P(\beta^*) \subseteq \mathcal X \cap P(\beta^*) \subseteq 
Sk(\beta^*) \subseteq N^*,
$$
so $C \cap P(\beta^*) \subseteq N^* \cap C$. 

Let $X := N^* \cap C$. 
Then $X$ is a finite subset of $N^*$, and so is in $N^*$.
Also $X \in \p$. 
Since $E$ is a dense subset of $N^* \cap \p$, we can fix $Y \le X$ in $E$. 
Since 
$$
Y \in E \subseteq N^* \cap \p \subseteq Sk(\beta^*),
$$
we have that $Y \in Sk(\beta^*)$. 
We will prove that $Y$ is compatible with $C$, which completes the proof.

We apply Proposition 3.11. 
We have that $C$ is adequate, $\beta^* \in \Gamma$, and $C$ is 
$\beta^*$-closed. 
Also, $Y$ is adequate, and 
$$
C \cap P(\beta^*) = N^* \cap C = X \subseteq Y 
\subseteq P(\beta^*).
$$
By Proposition 3.11, $Y \cup C$ is adequate. 
So $Y \cup C$ is in $\p$ and is below $Y$ and $C$, which proves that 
$Y$ and $C$ are compatible.
\end{proof}

Note that $\p$ has size $\omega_2$, and so preserves cardinals larger 
than $\omega_2$ as well.

\section{Adding a Function}

In this section we define a forcing poset for adding a generic 
function from $\omega_2$ to $\omega_2$ using adequate sets of 
models as side conditions.

We assume for the remainder of this section that $\Gamma = \Lambda$. 
It follows from Proposition 2.12 that if $\{ M, N \}$ is adequate, then 
$R_M(N) \subseteq \Gamma$.

\begin{definition}
Let $\p$ be the forcing poset whose conditions are pairs $(f,A)$ satisfying:
\begin{enumerate}
\item $f$ is a finite partial function from $\omega_2$ to $\omega_2$;
\item $A$ is a finite adequate set;
\item for all $M \in A$ and $\alpha \in \dom(f)$, 
if $M \cap [\alpha,f(\alpha)] \ne \emptyset$, then 
$\alpha, f(\alpha) \in M$.\footnote{For ordinals $\alpha$ and $\beta$, 
if we let $\alpha'$ be the smaller and $\alpha''$ 
the larger of $\alpha$ and $\beta$, then $[\alpha,\beta]$ denotes the closed 
interval $[\alpha',\alpha'']$.}
\end{enumerate}
Let $(g,B) \le (f,A)$ if $A \subseteq B$ and $f \subseteq g$.
\end{definition}

If $p = (f,A)$, we will write $f_p := f$ and $A_p := A$. 
It is easy to see that 
if $(f,A)$ is a condition, $f' \subseteq f$, and $A' \subseteq A$, then 
$(f',A')$ is a condition.

Let $\dot F$ be a $\p$-name for the set 
$$
\bigcup \{ f : \exists p \in \dot G \ f = f_p \}.
$$
Note that for any ordinal $\alpha < \omega_2$ and any condition 
$(f,A)$, we can extend $(f,A)$ to a condition $(g,B)$ which includes $\alpha$ 
in the domain of $g$. 
For example, let $g := f \cup \{ \langle \alpha,\alpha \rangle \}$ 
and $B := A$. 
Consequently, $\p$ forces that 
$\dot F$ is a total function from $\omega_2$ to $\omega_2$.

\bigskip

We will show that $\p$ preserves $\omega_1$ and $\omega_2$. 
Note that since $\p$ has size $\omega_2$, it preserves all cardinals 
larger than $\omega_2$ as well.

\begin{proposition}
The forcing poset $\p$ is strongly proper on a stationary set. 
In particular, $\p$ preserves $\omega_1$.
\end{proposition}

\begin{proof}
Fix $\theta > \omega_2$ regular. 
Let $N^*$ be a countable elementary substructure of $H(\theta)$ satisfying 
that $\p, \pi, \mathcal X \in N^*$ 
and $N := N^* \cap \omega_2 \in \mathcal X$. 
Note that since $\mathcal X$ is stationary, 
there are stationarily many such $N^*$ in 
$P_{\omega_1}(H(\theta))$. 
To prove the proposition, it suffices to show that 
every condition in $N^* \cap \p$ has a strongly $(N^*,\p)$-generic extension.

Observe that since $\pi \in N^*$ and $\pi : \omega_2 \to H(\omega_2)$ 
is a bijection, by elementarity we have that 
$$
N^* \cap H(\omega_2) = \pi[N^* \cap \omega_2] = \pi[N] = Sk(N),
$$
where the last equality holds by Lemma 1.3 
and the fact that 
$N \in \mathcal X$ implies that $Sk(N) \cap \omega_2 = N$. 
In particular, $N^* \cap \p \subseteq Sk(N)$.

Fix $p \in N^* \cap \p$. 
Then as just noted, $p \in Sk(N)$. 
Define 
$$
q := (f_p,A_p \cup \{ N \}).
$$
It is trivial to see that $q$ is a condition, using Proposition 3.5, 
and clearly $q \le p$. 
We will prove that $q$ is strongly $(N^*,\p)$-generic, which finishes the proof. 
So fix a set $D$ which is a dense subset of 
$N^* \cap \p$, and we will show that $D$ is 
predense below $q$.

\bigskip

Let $r \le q$ be given. 
Our goal is to find a condition in $D$ which is compatible with $r$. 
First let us extend $r$ to prepare for intersecting with the model $N^*$. 
Define $s$ so that $f_s := f_r$ and 
$$
A_s := A_r \cup \{ M \cap \beta_{M,N} : M \in A_r, \ 
M \cap \beta_{M,N} \in Sk(N) \}.
$$
We claim that $s$ is a condition. 
Requirement (1) in the definition of $\p$ is trivial. 
For (2), $A_s$ is adequate by Proposition 3.4. 

(3) Consider a model in $A_s \setminus A_r$ and $\alpha \in \dom(f_r)$. 
Then by definition this model has the form $M \cap \beta_{M,N}$, where 
$M \in A_r$ and $M \cap \beta_{M,N} \in Sk(N)$. 
Assume that 
$$
(M \cap \beta_{M,N}) \cap [\alpha,f_r(\alpha)] \ne \emptyset.
$$
We will show that $\alpha$ and $f_r(\alpha)$ are in $M \cap \beta_{M,N}$. 
Let $\alpha'$ be the smaller and $\alpha''$ the larger of 
$\alpha$ and $f_r(\alpha)$. 
Since $M \cap \beta_{M,N}$ meets the interval $[\alpha',\alpha'']$, 
clearly $\alpha' < \beta_{M,N}$. 

Since $M \cap \beta_{M,N}$ intersects the interval 
$[\alpha,f_r(\alpha)]$, obviously $M$ does as well. 
As $r$ is a condition, it follows that $\alpha'$ and $\alpha''$ are in $M$. 
But we observed above that $\alpha' < \beta_{M,N}$. 
Hence $\alpha' \in M \cap \beta_{M,N}$.

To show that $\alpha'' \in M \cap \beta_{M,N}$, it suffices 
to show that $\alpha'' < \beta_{M,N}$. 
Since $M \cap \beta_{M,N} \in Sk(N)$, 
it follows that $\alpha' \in N$. 
Therefore $N \cap [\alpha,f_r(\alpha)] \ne \emptyset$. 
Since $N \in A_r$ and $r$ is a condition, we have that $\alpha'' \in N$. 
Therefore $\alpha'' \in M \cap N \subseteq \beta_{M,N}$, so 
$\alpha'' < \beta_{M,N}$. 
This completes the proof of (3), and with it the proof that $s$ is a condition.

We will show that there is a condition in $D$ which is compatible with $s$. 
Since $s \le r$, this implies that there is a condition in $D$ which is 
compatible with $r$, which finishes the proof.

\bigskip

Define $u$ by 
$$
u := (f_s \cap Sk(N),A_s \cap Sk(N)).
$$
Note that $u \in N^* \cap \p \subseteq Sk(N)$. 
Define 
$$
R(N) := \bigcup \{ R_M(N) : M \in A_s \}.
$$
Then $R(N)$ is a finite subset of $N$, and therefore is in $N^*$. 
So we have that $N \in \mathcal X$, $u \in Sk(N)$, and $R(N) \subseteq N$. 
Since $\mathcal X \in N^*$, by the elementarity of $N^*$ 
we can fix $K \in N^*$ satisfying that $K \in \mathcal X$, 
$u \in Sk(K)$, and $R(N) \subseteq K$.

\bigskip

Define $v$ by letting $f_v := f_u$, and 
$$
A_v := A_u \cup \{ K \} \cup \{ K \cap \zeta : \zeta \in R(N) \}.
$$
Note that $v$ is in $N^*$. 
We claim that $v$ is a condition. 
Requirement (1) in the definition of $\p$ is trivial. 
For (2), since $u \in Sk(K)$, $A_u \subseteq Sk(K)$; 
so the set $A_v$ is adequate by Lemmas 3.4 and 3.5.

It remains to prove requirement (3) in the definition of $\p$. 
The proof will take some time. 
Let $\alpha \in \dom(f_v)$. 
Recall that $f_v = f_u = f_s \cap Sk(N)$. 
We need to show that any model in $A_v$ which meets the interval 
$[\alpha,f_v(\alpha)]$ contains $\alpha$ and $f_v(\alpha)$. 
Since $f_v = f_u$ and $u$ is a condition, clearly this 
requirement is satisfied for models in $A_u$. 
So it suffices to show that the requirement is satisfied by 
$K$ and $K \cap \zeta$, for 
all $\zeta \in R(N)$.

Since $u$ is in $Sk(K)$, so is $f_u = f_v$. 
Hence $\alpha$ and $f_v(\alpha) = f_u(\alpha)$ are in $K$. 
So $K$ satisfies the requirement.

Consider a model $K \cap \zeta$, where $\zeta \in R(N)$. 
By the definition of $R(N)$, fix $M \in A_s$ such that $\zeta \in R_M(N)$. 
Suppose that 
$$
(K \cap \zeta) \cap [\alpha,f_v(\alpha)] \ne \emptyset.
$$
We will show that $\alpha$ and $f_v(\alpha)$ are in $K \cap \zeta$. 
Since $\zeta \in R_M(N)$, by the definition of remainder points, 
$$
\beta_{M,N} \le \zeta.
$$
Let $\alpha'$ be the smaller and $\alpha''$ the larger of 
$\alpha$ and $f_v(\alpha)$. 
Then clearly $\alpha' < \zeta$, so $\alpha' \in K \cap \zeta$. 
Since $\alpha'' \in K$ as observed above, we will be done if we can show 
that $\alpha'' < \zeta$.

Suppose for a contradiction that $\zeta \le \alpha''$. 
Then we have that 
$$
\alpha' < \zeta \le \alpha''.
$$
Since $\beta_{M,N} \le \zeta$, it follows that 
$$
\beta_{M,N} \le \alpha''.
$$

We claim that 
$$
M \cap [\alpha',\alpha''] = \emptyset.
$$
If $M \cap [\alpha',\alpha''] \ne \emptyset$, then since $f_v(\alpha) = f_s(\alpha)$, 
$s$ is a condition, and $M \in A_s$, it follows that 
$\alpha'' \in M$. 
But this is impossible, since then we would have that 
$$
\alpha'' \in M \cap N \subseteq \beta_{M,N} \le \zeta,
$$
which contradicts our assumption that $\zeta \le \alpha''$.

We will get a contradiction to our assumption that $\zeta \le \alpha''$ 
by separately considering the two cases that 
$\beta_{M,N} \le \alpha'$ and $\alpha' < \beta_{M,N}$.

First, assume that $\beta_{M,N} \le \alpha'$. 
Recall that $\zeta \in R_M(N)$. 
Since the ordinals $\alpha' < \zeta$ are in $N$, it obviously cannot be 
the case that $\zeta = \min(N \setminus \beta_{M,N})$. 
So by the definition of remainder points, 
there is $\gamma \ge \beta_{M,N}$ in $M$ 
such that $\zeta = \min(N \setminus \gamma)$. 
Since $\alpha' \in N$, 
it must be the case that $\alpha' < \gamma < \zeta$. 
Hence $M$ meets the interval $[\alpha',\alpha'']$, which contradicts the claim 
above that $M \cap [\alpha',\alpha''] = \emptyset$.

In the second case, assume that $\alpha' < \beta_{M,N}$. 
Then $\alpha' \in (N \cap \beta_{M,N}) \setminus M$, which implies that 
$M \cap \beta_{M,N} \in Sk(N)$, since the other two kinds of comparisons of 
$M$ and $N$ would imply that $\alpha' \in M$. 
By the definition of $R_M(N)$, there is $\gamma \ge \beta_{M,N}$ 
in $M$ such that $\zeta = \min(N \setminus \gamma)$. 
Since $\beta_{M,N} > \alpha'$, this implies that $\gamma$ is in the interval 
$[\alpha',\alpha'']$, which again contradicts that 
$M \cap [\alpha',\alpha''] = \emptyset$. 
This contradiction shows that $\alpha'' < \zeta$, which completes the 
proof that $v$ is a condition. 

\bigskip

Since $D$ is dense in $N^* \cap \p$ and $v \in N^* \cap \p$, 
we can fix $w \le v$ in $D$. 
We will show that $w$ and $s$ are compatible, which finishes the proof. 
Since $D \subseteq \p \cap N^* \subseteq Sk(N)$, we have that 
$w \in \p \cap Sk(N)$. 
Define 
$$
z := (f_w \cup f_s,A_w \cup A_s).
$$
We claim that $z$ is a condition. 
Then clearly $z \le w, s$ and we are done. 
We check requirements (1), (2), and (3) in the definition of $\p$.

\bigskip

(1) We show that $f_w \cup f_s$ is a function. 
Let $\alpha \in \dom(f_w) \cap \dom(f_s)$, and we will prove that 
$f_w(\alpha) = f_s(\alpha)$. 
Since $\alpha \in \dom(f_w)$ and $w \in N$, it follows that $\alpha \in N$. 
Hence $N \cap [\alpha,f_s(\alpha)] \ne \emptyset$, which implies that 
$\alpha, f_s(\alpha) \in N$, since $s$ is a condition. 
So the ordered pair $\langle \alpha,f_s(\alpha) \rangle$ is in 
$N^* \cap f_s$. 
But 
$$
N^* \cap f_s = Sk(N) \cap f_s = f_u = f_v \subseteq f_w.
$$
Therefore $f_w(\alpha) = f_s(\alpha)$.

\bigskip

(2) Since $A_s$ is $N$-closed, the set 
$A_z$ is adequate by Proposition 3.9.

\bigskip

(3) Let $M \in A_z$ and $\alpha \in \dom(f_z)$, and suppose that 
$M \cap [\alpha,f_z(\alpha)] \ne \emptyset$. 
We will show that $\alpha$ and $f_z(\alpha)$ are in $M$. 
Since $w$ and $s$ are conditions, it suffices to consider the cases that 
(A) $M \in A_w$ and $\alpha \in \dom(f_s)$, or 
(B) $M \in A_s$ and $\alpha \in \dom(f_w)$. 

\bigskip

(A) $M \in A_w$ and $\alpha \in \dom(f_s)$. 
As $w \in Sk(N)$, also $M \in Sk(N)$. 
So $M \subseteq N$. 
Since $M$ meets the interval $[\alpha,f_s(\alpha)]$ and $M \subseteq N$, 
also $N$ meets the interval $[\alpha,f_s(\alpha)]$. 
Since $s$ is a condition, it follows that $\alpha$ and $f_s(\alpha)$ are in $N$. 
Hence the pair $\langle \alpha, f_s(\alpha) \rangle$ is in $f_s \cap Sk(N)$. 
But 
$$
f_s \cap Sk(N) \subseteq f_u = f_v \subseteq f_w.
$$
So $f_s(\alpha) = f_w(\alpha)$. 
Since $w$ is a condition and $M \in A_w$, it follows that 
$\alpha, f_s(\alpha) \in M$.

\bigskip

(B) $M \in A_s$ and $\alpha \in \dom(f_w)$. 
Then $\alpha$ and $f_w(\alpha)$ are in $N$. 
Let $\alpha'$ be the smaller and $\alpha''$ the larger of 
$\alpha$ and $f_w(\alpha)$.

Suppose that there is $\gamma \in M \cap [\alpha,f_w(\alpha)]$ such that 
$\gamma \ge \beta_{M,N}$. 
We will get a contradiction from this assumption. 
Since $\alpha' \le \gamma$, $\alpha' \in N$, $\gamma \in M$, and 
$\gamma \ge \beta_{M,N}$, it follows that $\alpha' < \gamma$ 
by Proposition 2.6. 
Let $\zeta = \min(N \setminus \gamma)$. 
Then $\zeta \in R_M(N)$ and $\zeta \in (\alpha',\alpha'']$. 
Since $R(N) \subseteq K$, we have that $\zeta \in K$. 
Therefore 
$$
K \cap [\alpha,f_w(\alpha)] \ne \emptyset.
$$
Since $K \in A_w$ and $w$ is a condition, it follows that 
$\alpha'$ and $\alpha''$ are in $K$. 
But now $\alpha' < \zeta$, so $\alpha' \in K \cap \zeta$. 
Hence 
$$
(K \cap \zeta) \cap [\alpha,f_w(\alpha)] \ne \emptyset.
$$
Since $K \cap \zeta \in A_w$ and $w$ is a condition, it follows that 
$\alpha'' \in K \cap \zeta$, and in particular, $\alpha'' < \zeta$. 
But this is impossible since $\zeta \le \alpha''$.

It follows that the nonempty intersection 
$M \cap [\alpha,f_w(\alpha)]$ is a subset of $\beta_{M,N}$. 
So clearly 
$$
(M \cap \beta_{M,N}) \cap [\alpha,f_w(\alpha)] \ne \emptyset.
$$
Note that this also implies that $\alpha' < \beta_{M,N}$. 

If $M \cap \beta_{M,N} \in Sk(N)$, then 
$$
M \cap \beta_{M,N} \in A_s \cap Sk(N) = A_u \subseteq A_v \subseteq A_w,
$$
so $M \cap \beta_{M,N} \in A_w$. 
Since $w$ is a condition, 
$\alpha$ and $f_w(\alpha)$ are in 
$M \cap \beta_{M,N}$, and hence in $M$. 
So in this case we are done. 

Otherwise $N \cap \beta_{M,N}$ is either equal to 
$M \cap \beta_{M,N}$ or in $Sk(M)$. 
In either case, $N \cap \beta_{M,N} \subseteq M$. 
If $\alpha'' < \beta_{M,N}$, then 
$\alpha$ and $f_w(\alpha)$ are both in 
$N \cap \beta_{M,N}$, and hence in $M$, and we are done. 
So assume that $\alpha' < \beta_{M,N} \le \alpha''$, and we will get 
a contradiction.

Let $\zeta = \min(N \setminus \beta_{M,N})$. 
Then $\zeta \in R_M(N)$, and $\alpha' < \zeta \le \alpha''$. 
Since $R(N) \subseteq K$, it follows that 
$\zeta \in K$, and hence $K$ meets the interval 
$[\alpha,f_w(\alpha)]$. 
Since $w$ is a condition, it follows that $\alpha' \in K$. 
So $\alpha' \in K \cap \zeta$, which implies that 
$K \cap \zeta$ meets the interval $[\alpha,f_w(\alpha)]$. 
Since $w$ is a condition and $K \cap \zeta \in A_w$, it follows that 
$\alpha'' \in K \cap \zeta$. 
In particular, $\alpha'' < \zeta$. 
But this contradicts that $\zeta \le \alpha''$.
\end{proof}

\begin{proposition}
The forcing poset $\p$ preserves $\omega_2$.
\end{proposition}

\begin{proof}
Let $\theta > \omega_2$ be regular. 
Fix $N^* \prec H(\theta)$ of size $\omega_1$ such that 
$\p, \pi, \mathcal X \in N^*$ and 
$\beta^* := N^* \cap \omega_2 \in \Gamma$. 
Note that since $\Gamma$ is stationary in $\omega_2$, there are 
stationarily many such $N^*$ in $P_{\omega_2}(H(\theta))$. 
So to prove the proposition, it suffices to show that any condition in 
$N^* \cap \p$ 
has a strongly $(N^*,\p)$-generic extension. 
Fix $p \in N^* \cap \p$.

Observe that since $\pi \in N^*$ and 
$\pi : \omega_2 \to H(\omega_2)$ is a bijection, 
$$
N^* \cap H(\omega_2) = \pi[N^* \cap \omega_2] = \pi[\beta^*] = 
Sk(\beta^*),
$$
where the last equality holds by Lemma 1.3 
and the fact that $\beta^* \in \Gamma$ implies that 
$Sk(\beta^*) \cap \omega_2 = \beta^*$. 
In particular, $N^* \cap \p \subseteq Sk(\beta^*)$.

Fix $K \in \mathcal X$ with $\beta^* \in K$ and $p \in Sk(K)$.  
Then $p \in Sk(K) \cap Sk(\beta^*) = Sk(K \cap \beta^*)$. 
Define $q$ by letting $f_q := f_p$, and 
$$
A_q := A_p \cup \{ K \} \cup \{ K \cap \beta^* \}.
$$
Note that $A_q$ is adequate by Proposition 3.5 applied to $A_p$ and $K$ 
and Proposition 3.4 applied to $A_p \cup \{ K \}$ and $\beta^*$. 
It follows that $q$ is a condition, and easily $q \le p$.

We claim that $q$ is strongly $(N^*,\p)$-generic. 
So fix a set $D$ which is a dense subset of $N^* \cap \p$, 
and we will show that 
$D$ is predense below $q$.
Fix $r \le q$, 
and we will show that $r$ is compatible with some condition in $D$. 

We claim that if $\alpha \in \dom(f_r)$ and one of $\alpha$ or $f_r(\alpha)$ is 
below $\beta^*$, then they are both below $\beta^*$. 
For let $\alpha'$ be the smaller and $\alpha''$ the larger of 
$\alpha$ and $f_r(\alpha)$, and 
assume that $\alpha' < \beta^*$. 
Suppose for a contradiction that $\alpha'' \ge \beta^*$. 
Then since $\beta^* \in K$, 
$$
K \cap [\alpha,f_r(\alpha)] \ne \emptyset.
$$
So $\alpha, f_r(\alpha) \in K$, since $r$ is a condition. 
Hence $\alpha' \in K \cap \beta^*$. 
But then 
$$
(K \cap \beta^*) \cap [\alpha,f_r(\alpha)] \ne \emptyset.
$$
Since $r$ is a condition, we have that 
$\alpha'' \in K \cap \beta^*$. 
In particular, $\alpha'' < \beta^*$, 
which contradicts that $\alpha'' \ge \beta^*$.

\bigskip

We extend $r$ to $s$ to prepare for intersecting with $N^*$. 
Define $s$ by letting $f_s := f_r$ and 
$$
A_s := A_r \cup \{ M \cap \beta^* : M \in A_r \}.
$$
We claim that $s$ is a condition. 
Requirements (1) and (2) in the definition of $\p$ are easy, using 
Proposition 3.4. 
For (3), suppose that $\alpha \in \dom(f_r)$, $M \in A_r$, and 
$$
(M \cap \beta^*) \cap [\alpha,f_r(\alpha)] \ne \emptyset.
$$
Then obviously $M \cap [\alpha,f_r(\alpha)] \ne \emptyset$, so 
$\alpha$ and $f_r(\alpha)$ are in $M$ since $r$ is a condition. 
Let $\alpha'$ be the smaller and $\alpha''$ the larger of 
$\alpha$ and $f_r(\alpha)$. 
Since $M \cap \beta^*$ meets the interval $[\alpha,f_r(\alpha)]$, 
clearly $\alpha' < \beta^*$. 
By the claim in the preceding paragraph, 
it follows that $\alpha'' < \beta^*$. 
So $\alpha,f_r(\alpha) \in M \cap \beta^*$.

We will find a condition in $D$ which is compatible with $s$. 
Since $s \le r$, it follows that there is a condition in $D$ which is 
compatible with $r$, completing the proof.

\bigskip

Let 
$$
v := (f_s \cap Sk(\beta^*), A_s \cap Sk(\beta^*)).
$$
So $f_v = f_s \cap (\beta^* \times \beta^*)$, and by Proposition 1.11, 
$A_v = A_s \cap P(\beta^*)$. 
Clearly $v$ is a condition and $v$ is in $N^*$.

Since $D$ is dense in $N^* \cap \p$, fix $w \le v$ in $D$. 
Then $w \in N^* \cap \p \subseteq Sk(\beta^*)$. 
We will show that $w$ is compatible with $s$.

Let 
$$
z := (f_w \cup f_s,A_w \cup A_s).
$$
We will prove that $z$ is a condition. 
Then clearly $z \le w, s$, which completes the proof. 
We check requirements (1), (2), and (3) in the definition of $\p$.

\bigskip

(1) Let $\alpha \in \dom(f_w) \cap \dom(f_s)$. 
Then $\alpha < \beta^*$. 
Thus $f_s(\alpha) < \beta^*$ by the claim above. 
Hence 
$$
\langle \alpha, f_s(\alpha) \rangle \in f_s \cap Sk(\beta^*) = f_v 
\subseteq f_w.
$$
So $\langle \alpha, f_s(\alpha) \rangle \in f_w$, that is, 
$f_s(\alpha) = f_w(\alpha)$. 
This shows that $f_w \cup f_s$ is a function.

\bigskip

(2) $A_z$ is adequate by Proposition 3.11, since $A_s$ is $\beta^*$-closed.

\bigskip

(3) Let $M \in A_s$ and $\alpha \in \dom(f_w)$, and assume that 
$$
M \cap [\alpha,f_w(\alpha)] \ne \emptyset.
$$
We will show that $\alpha$ and $f_w(\alpha)$ are in $M$. 
Since $w \in N^*$, the ordinals 
$\alpha$ and $f_w(\alpha)$ are less than $\beta^*$. 
So 
$$
(M \cap \beta^*) \cap [\alpha,f_w(\alpha)] \ne \emptyset.
$$
But 
$$
M \cap \beta^* \in A_s \cap Sk(\beta^*) = A_v \subseteq A_w.
$$
So $M \cap \beta^* \in A_w$. 
Since $w$ is a condition, 
the ordinals $\alpha$ and $f_w(\alpha)$ are in 
$M \cap \beta^*$, and hence in $M$.

Now let $M \in A_w$ and $\alpha \in \dom(f_s)$, and suppose that 
$$
M \cap [\alpha,f_s(\alpha)] \ne \emptyset.
$$
We will show that $\alpha$ and $f_s(\alpha)$ are in $M$. 
Since $M \subseteq \beta^*$, the smaller of $\alpha$ and $f_s(\alpha)$ is 
below $\beta^*$. 
By the claim above, this implies that $\alpha$ and $f_s(\alpha)$ are both 
below $\beta^*$. 
Hence 
$$
\langle \alpha, f_s(\alpha) \rangle \in f_s \cap Sk(\beta^*) = f_v \subseteq f_w.
$$
Therefore $f_s(\alpha) = f_w(\alpha)$. 
Since $M \in A_w$ and $w$ is a condition, we have that 
$\alpha$ and $f_s(\alpha) = f_w(\alpha)$ are in $M$.
\end{proof}

\section{Adding a nonreflecting stationary set}

We now give an example of a forcing poset using 
adequate sets of models as side conditions for adding a more complex object. 
We define a forcing poset which adds a stationary subset of 
$\omega_2 \cap \cof(\omega)$ with finite conditions 
which does not reflect.\footnote{The classical way of adding a nonreflecting set is by initial segments, 
ordered by end-extension.}

\begin{definition}
Let $\p$ be the forcing poset whose conditions are triples $(a,x,A)$ satisfying:
\begin{enumerate}
\item $a$ is a finite subset of $\omega_2 \cap \cof(\omega)$;
\item $x$ is a finite set of triples $\langle \alpha,\gamma,\beta \rangle$, where 
$\alpha \in \Gamma$ and $\gamma < \beta < \alpha$;
\item $A$ is a finite adequate set;
\item if $\langle \alpha,\gamma,\beta \rangle$ and 
$\langle \alpha,\gamma',\beta' \rangle$ are 
distinct triples in $x$, 
then $[\gamma,\beta) \cap [\gamma',\beta') = \emptyset$;
\item if $\xi \in a$, $M \in A$, $\sup(M \cap \xi) = \xi$, and 
$M \setminus \xi \ne \emptyset$, then $\xi \in M$;
\item suppose that $M \in A$, $\alpha \in M$, 
and $\langle \alpha,\gamma,\beta \rangle \in x$; 
if $M \cap [\gamma,\beta] \ne \emptyset$, then $\gamma,\beta \in M$; 
if $M \cap [\gamma,\beta] = \emptyset$, then 
$\sup(M \cap \alpha) < \gamma$.
\end{enumerate}
Let $(b,y,B) \le (a,x,A)$ if $a \subseteq b$, $x \subseteq y$, and 
$A \subseteq B$.
\end{definition}

If $p = (a,x,A)$ is a condition, we write $a_p := a$, $x_p := x$, and $A_p := A$.

\bigskip

We give some motivation for the definition. 
The first component of a condition approximates a generic stationary subset of 
$\omega_2 \cap \cof(\omega)$. 
Let $\dot S$ be a $\p$-name such that $\p$ forces 
$$
\dot S = \{ \xi : \exists p \in \dot G \ \xi \in a_p \}.
$$
For each $\alpha \in \Gamma$, let $\dot c_\alpha$ be a $\p$-name such 
that $\p$ forces
$$
\dot c_\alpha = \{ \gamma : \exists p \in \dot G \ \exists \beta \ 
\langle \alpha, \gamma, \beta \rangle \in x_p \}.
$$
We will show that $\dot c_\alpha$ is forced to be cofinal in $\alpha$. 
Property (5) in the definition of $\p$ 
will imply that $\dot S$ does not contain any limit points of 
$\dot c_\alpha$, and thus $\dot S \cap \alpha$ is nonstationary in $\alpha$.

\bigskip

We first prove that $\p$ preserves $\omega_1$ and $\omega_2$ 
and forces that $\dot S$ is stationary. 
Since $\p$ has size $\omega_2$, it also preserves cardinals larger than $\omega_2$. 
We then analyze the limit points of the $\dot c_\alpha$'s and show that 
$\dot S$ does not reflect.

Note that if $(a,x,A)$ is a condition, $M_1, \ldots, M_k \in A$, and 
$\beta_1, \ldots, \beta_k \in \Gamma$, then 
$(a,x,A \cup \{ M_1 \cap \beta_1, \ldots, M_k \cap \beta_k \})$ 
is a condition. 
For requirements (1)--(4) are immediate using Proposition 3.4, 
and (5) and (6) are preserved under taking 
initial segments of models.

\begin{proposition}
The forcing poset $\p$ is strongly proper on a stationary set, and forces that 
$\dot S$ is stationary.
\end{proposition}

\begin{proof}
Let $\dot E$ be a $\p$-name for a club subset of $\omega_2$. 
Fix a regular cardinal $\theta > \omega_2$ with 
$\p$ and $\dot E$ in $H(\theta)$. 
Let $N^*$ be a countable elementary substructure of 
$H(\theta)$ which contains $\p, \dot E, \pi$ and satisfies that 
$N := N^* \cap \omega_2 \in \mathcal X$. 
Note that since $\mathcal X$ is stationary, there are stationarily many such 
$N^*$ in $P_{\omega_1}(H(\theta))$. 

Observe that since $\pi \in N^*$ and $\pi : \omega_2 \to H(\omega_2)$ 
is a bijection, by elementarity we have that 
$$
N^* \cap H(\omega_2) = \pi[N^* \cap \omega_2] = \pi[N] = Sk(N),
$$
where the last equality holds by Lemma 1.3 
and the fact that 
$N \in \mathcal X$ implies that $Sk(N) \cap \omega_2 = N$. 
In particular, $N^* \cap \p \subseteq Sk(N)$.

Let $p \in N^* \cap \p$. 
We will find an extension of $p$ which is strongly $(N^*,\p)$-generic. 
Let $\xi^* := \sup(N \cap \omega_2)$. 
Define 
$$
q := (a_p \cup \{ \xi^* \}, x_p, A_p \cup \{ N \}).
$$
It is easy to check that $q$ is a condition, and $q \le p$. 
We will prove that $q$ is strongly $(N^*,\p)$-generic. 

If this argument is successful, then clearly $\p$ is strongly proper 
on a stationary set. 
Let us note that this argument also shows that $\p$ forces that 
$\dot S$ is stationary. 
For given a condition $p$, we can find $N^*$ as above such that 
$p \in N^*$. 
Let $q \le p$ be strongly $(N^*,\p)$-generic. 
Since $q$ is strongly $(N^*,\p)$-generic, by standard proper forcing facts, 
$q$ forces that $N^*[\dot G] \cap On = N \cap On$. 
As $\dot E \in N^*$, $q$ forces that 
$$
\xi^* = \sup(N^* \cap \omega_2) = \sup(N^*[\dot G] \cap \omega_2) \in \dot E.
$$
Since $q$ also forces that $\xi^* \in \dot S$, this shows that $q$ forces 
that $\dot E \cap \dot S$ is nonempty.

\bigskip
 
Towards proving that $q$ is strongly $(N^*,\p)$-generic, 
fix a set $D$ which is a dense subset of $N^* \cap \p$. 
We will show that 
$D$ is predense below $q$.
Let $r \le q$ be given, and we 
will find a condition in $D$ which is compatible with $r$.

We extend $r$ to prepare for intersecting with $N^*$. 
Define $s$ by letting $a_s := a_r$, $x_s := x_r$, and 
$$
A_s := A_r \cup \{ M \cap \beta_{M,N} : M \in A_r, \ 
M \cap \beta_{M,N} \in Sk(N) \}.
$$
Then $A_s$ is $N$-closed (see Definition 3.8). 
By the comments preceding the proposition, 
$s$ is a condition, and clearly $s \le r$. 
Since $s \le r$, we will be done if we can find a condition in $D$ which 
is compatible with $s$.

\bigskip

Define 
$$
u := (a_s \cap Sk(N),x_s \cap Sk(N),A_s \cap Sk(N)).
$$
Note that $u$ is in $\p \cap Sk(N)$, and clearly $s \le u$. 

Let $Z$ be the set of models in $A_u$ of the form 
$M \cap \beta_{M,N}$, where $M \in A_s$ and 
$M \setminus \beta_{M,N} \ne \emptyset$. 
Note that for such an $M$, the ordinal $\sup(M \cap \beta_{M,N})$ is not in $M$.  
For otherwise, as $\beta_{M,N}$ has cofinality $\omega_1$, 
$\sup(M \cap \beta_{M,N})$ would be in $M \cap \beta_{M,N}$, which 
is impossible since $M$ is closed under successors. 
The set $Z$ is in $N^*$ because it is a finite subset of $A_u$. 

The condition $s$ satisfies the property that $s \le u$, and for all $K \in Z$, 
there is $M \in A_s$ such that $K$ is a proper initial segment of $M$ and 
$\sup(K) \notin M$. 
By the elementarity of $N^*$, we can fix a condition $v \le u$ in $N^*$ 
such that for all $K \in Z$, there is $M \in A_v$ such that $K$ is a proper 
initial segment of $M$ and $\sup(K) \notin M$.

Since $D$ is dense in $N^* \cap \p$, we can fix $w \le v$ in $D$. 
We will show that $w$ and $s$ are compatible, which finishes the proof. 
As $D \subseteq N^* \cap \p \subseteq Sk(N)$, we have that 
$w \in \p \cap Sk(N)$. 
Define 
$$
z := (a_w \cup a_s,x_w \cup x_s,A_w \cup A_s).
$$
We claim that $z$ is a condition. 
Then clearly $z \le w, s$, and we are done. 
We verify that $z$ satisfies requirements (1)--(6) in the definition of $\p$.

\bigskip

(1) and (2) are immediate, and (3) follows 
from Proposition 3.9, since $A_s$ is $N$-closed.

\bigskip
 
(4) Let $\langle \alpha,\gamma,\beta \rangle \in x_w$ and 
$\langle \alpha,\gamma',\beta' \rangle \in x_s$ be distinct. 
Then $\alpha \in N$. 
If $N \cap [\gamma',\beta'] \ne \emptyset$, then $\gamma', \beta' \in N$ since $s$ is a condition. 
So in that case, 
$$
\langle \alpha,\gamma',\beta' \rangle \in x_s \cap Sk(N) = x_u 
\subseteq x_v \subseteq x_w.
$$
Hence $[\gamma,\beta) \cap [\gamma',\beta') = \emptyset$, since $w$ is a condition. 

Otherwise $N \cap [\gamma',\beta'] = \emptyset$. 
Since $\alpha \in N$ and $s$ is a condition, 
$\sup(N \cap \alpha) < \gamma'$. 
But $\beta \in N \cap \alpha$, so 
$\beta < \sup(N \cap \alpha) < \gamma'$. 
So clearly $[\gamma,\beta) \cap [\gamma',\beta') = \emptyset$.

\bigskip

(5) Suppose that $\xi \in a_s$, $M \in A_w$, 
$\sup(M \cap \xi) = \xi$, and $M \setminus \xi \ne \emptyset$. 
We will show that $\xi \in M$. 
Since $M \in Sk(N)$, $M \cap \xi$ is in $Sk(N)$, since it is an initial segment of $M$. 
So $\sup(M \cap \xi) = \xi$ is in $N$. 
Hence 
$$
\xi \in a_s \cap Sk(N) = a_u \subseteq a_v \subseteq a_w.
$$
Since $w$ is a condition, it follow that $\xi$ is in $M$.

Now assume that $\xi \in a_w$, $M \in A_s$, 
$\sup(M \cap \xi) = \xi$, and $M \setminus \xi \ne \emptyset$. 
We will prove that $\xi \in M$. 
Suppose for a contradiction that $\xi \notin M$. 
Since $\sup(M \cap \xi) = \xi$ and $\xi \in N$, it follows that 
$\xi < \beta_{M,N}$ by Proposition 2.6. 
But $\xi \in N \setminus M$. 
So the only comparison between $M$ and $N$ that 
is possible is that $M \cap \beta_{M,N}$ is in $Sk(N)$, since the other comparisions 
together with the fact that $\xi < \beta_{M,N}$ would imply that $\xi \in M$. 
Therefore 
$$
M \cap \beta_{M,N} \in A_s \cap Sk(N) = A_u \subseteq A_v \subseteq A_w.
$$
So $M \cap \beta_{M,N} \in A_w$.

If $\min(M \setminus \xi) < \beta_{M,N}$, then 
$\xi \in M \cap \beta_{M,N}$ since $w$ is a condition, which is a contradiction. 
Therefore $\min(M \setminus \xi) > \beta_{M,N}$. 
So $M \cap \beta_{M,N} \in Z$. 
It easily follows that $M \cap \beta_{M,N} = M \cap \xi$, and hence 
$$
\sup(M \cap \beta_{M,N}) = \sup(M \cap \xi) = \xi.
$$
By the choice of $v$ and the fact that $M \cap \beta_{M,N}$ is in $Z$, 
there is $L \in A_w$ such that 
$M \cap \beta_{M,N}$ is a proper initial segment of $L$ and 
$\sup(M \cap \beta_{M,N}) = \xi$ is not in $L$. 
But then $L \in A_w$, $\xi \in a_w$, $\sup(L \cap \xi) = \xi$, 
and $L \setminus \xi$ is nonempty. 
Since $w$ is a condition, $\xi \in L$, which is a contradiction.

\bigskip

(6) Suppose that $M \in A_w$, $\alpha \in M$, and 
$\langle \alpha,\gamma,\beta \rangle \in x_s$. 
Since $M \in Sk(N)$, $\alpha \in N$. 
Suppose that $N \cap [\gamma,\beta] \ne \emptyset$. 
Then $\gamma, \beta \in N$, since $s$ is a condition. 
Hence 
$$
\langle \alpha,\gamma,\beta \rangle \in x_s \cap Sk(N) = x_u 
\subseteq x_v \subseteq x_w.
$$
Since $w$ is a condition, it follows that 
$M$ and $\langle \alpha,\gamma,\beta \rangle$ satisfy requirement (6). 

Suppose on the other hand that $N \cap [\gamma,\beta] = \emptyset$. 
Then since $s$ is a condition, 
$\sup(N \cap \alpha) < \gamma$. 
Hence 
$$
\sup(M \cap \alpha) < \sup(N \cap \alpha) < \gamma,
$$
so again (6) is satisfied.

Now suppose that $M \in A_s$, $\alpha \in M$, 
and $\langle \alpha,\gamma,\beta \rangle \in x_w$. 
Then $\alpha \in M \cap N$, so $\alpha < \beta_{M,N}$ 
by Proposition 2.6. 
If $N \cap \beta_{M,N}$ is either equal to 
$M \cap \beta_{M,N}$ or in $Sk(M)$, 
then $N \cap \beta_{M,N} \subseteq M$, and hence $\gamma, \beta \in M$, 
which proves (6).

Assume that $M \cap \beta_{M,N}$ is in $Sk(N)$. 
Then 
$$
M \cap \beta_{M,N} \in A_s \cap Sk(N) = A_u \subseteq A_v 
\subseteq A_w.
$$
If $M \cap [\gamma,\beta] \ne \emptyset$, it follows that 
$$
(M \cap \beta_{M,N}) \cap [\gamma,\beta] \ne \emptyset,
$$
since $\gamma < \beta < \alpha < \beta_{M,N}$. 
Since $w$ is in a condition, $\gamma,\beta$ are in $M \cap \beta_{M,N}$, 
and hence in $M$. 
Otherwise $M \cap [\gamma,\beta] = \emptyset$. 
Then obviously 
$$
(M \cap \beta_{M,N}) \cap [\gamma,\beta] = \emptyset.
$$
So $\sup((M \cap \beta_{M,N}) \cap \alpha) < \gamma$. 
But since $\alpha < \beta_{M,N}$, 
it follows that 
$$
(M \cap \beta_{M,N}) \cap \alpha = M \cap \alpha,
$$
so $\sup(M \cap \alpha) < \gamma$.
\end{proof}

\begin{proposition}
The forcing poset $\p$ is $\omega_2$-c.c.
\end{proposition}

\begin{proof}
We will use Lemma 4.3. 
Let $\theta > \omega_2$ be regular. 
Fix $N^* \prec H(\theta)$ of size $\omega_1$ such that 
$\p, \pi, \mathcal X \in N^*$ and 
$\beta^* := N^* \cap \omega_2 \in \Gamma$. 
Note that since $\Gamma$ is stationary, there are stationarily many such 
models $N^*$ in $P_{\omega_2}(H(\theta))$. 

Observe that as $\pi \in N^*$ and $\pi : \omega_2 \to H(\omega_2)$ 
is a bijection, by elementarity we have that 
$$
N^* \cap H(\omega_2) = \pi[N^* \cap \omega_2] = \pi[\beta^*] = Sk(\beta^*),
$$
where the last equality holds by Lemma 1.3 
and the fact that 
$\beta^* \in \Gamma$ implies that $Sk(\beta^*) \cap \omega_2 = \beta^*$.   
In particular, $N^* \cap \p \subseteq Sk(\beta^*)$.

We will prove that the empty condition is strongly $(N^*,\p)$-generic. 
By Lemma 4.3, this implies that $\p$ is $\omega_2$-c.c. 
So fix a set $D$ which is a dense subset of $N^* \cap \p$, 
and we will show that $D$ is predense in $\p$.

Let $r \in \p$ be given. 
We will find a condition in $D$ which is compatible with $r$, which 
completes the proof. 
We extend $r$ to prepare for intersecting with $N^*$. 
Define $s$ so that $a_s := a_r$, $x_s := x_r$, and 
$$
A_s := A_r \cup \{ M \cap \beta^* : M \in A_r \}.
$$
Then easily $s$ is a condition, and $s \le r$. 
Since $s \le r$, we will be done if we can find a condition in $D$ which 
is compatible with $s$.

Define 
$$
u := (a_s \cap Sk(\beta^*),x_s \cap Sk(\beta^*),A_s \cap Sk(\beta^*)).
$$
In other words, $a_u := x_s \cap \beta^*$, 
$x_u := x_s \cap (\beta^*)^3$, and by Proposition 1.11, 
$A_u := A_s \cap P(\beta^*)$. 
Let $Z$ be the set of models in $A_u$ of the form 
$M \cap \beta^*$, where $M \in A_s$ and $M \setminus \beta^*$ is nonempty. 
Since $Z$ is finite, it is a member of $N^*$.

The condition $s$ satisfies that $s \le u$, and for all $K \in Z$, there is 
$M \in A_s$ such that $K$ is a proper initial segment of $M$ and 
$\sup(K) \notin M$. 
By elementarity, we can fix $v \le u$ in $N^*$ satisfying that for all $K \in Z$, 
there is $M \in A_v$ such that $K$ is a proper initial segment of $M$ 
and $\sup(K) \notin M$.

Since $D$ is dense in $N^* \cap \p$, fix $w \le v$ in $D$. 
We will show that $w$ and $s$ are compatible, which finishes 
the proof.

Since $D \subseteq \p \cap N^* \subseteq Sk(\beta^*)$, we have 
that $w \in Sk(\beta^*)$. 
Define 
$$
z := (x_w \cup x_s,x_w \cup x_s,A_w \cup A_s).
$$
We will prove that $z$ is a condition. 
Then clearly $z \le w, s$, and we are done. 
We verify requirements (1)--(6) in the definition of $\p$.

\bigskip

(1) and (2) are immediate, 
and (3) follows from Proposition 3.11 using 
the fact that $A_s$ is $\beta^*$-closed.

\bigskip

(4) Let $\langle \alpha,\gamma,\beta \rangle \in x_w$ 
and $\langle \alpha,\gamma',\beta' \rangle \in x_s$ be distinct. 
Then $\alpha < \beta^*$. 
So $\gamma' < \beta' < \alpha < \beta^*$. 
Hence 
$$
\langle \alpha,\gamma',\beta' \rangle \in x_s \cap Sk(\beta^*) = x_u 
\subseteq x_v \subseteq x_w.
$$
So $[\gamma,\beta) \cap [\gamma',\beta') = \emptyset$, 
since $w$ is a condition.

\bigskip

(5) Suppose that $M \in A_w$, $\xi \in a_s$, 
$\sup(M \cap \xi) = \xi$, and $M \setminus \xi \ne \emptyset$. 
Then $M \in N^*$, so that $\sup(M) < \beta^*$. 
Therefore $\xi < \beta^*$. 
So 
$$
\xi \in a_s \cap Sk(\beta^*) = a_u \subseteq a_v \subseteq a_w.
$$
Since $w$ is a condition, it follows that $\xi \in M$.

Now assume that $M \in A_s$, $\xi \in a_w$, 
$\sup(M \cap \xi) = \xi$, and 
$M \setminus \xi \ne \emptyset$. 
We need to show that $\xi \in M$. 
Since $\xi \in a_w \subseteq N^*$, we have that $\xi < \beta^*$. 
Also by Proposition 1.11, 
$$
M \cap \beta^* \in A_s \cap Sk(\beta^*) = A_u \subseteq A_v 
\subseteq A_w.
$$
So if $(M \cap \beta^*) \setminus \xi$ is nonempty, then 
$\xi \in M \cap \beta^*$ since $w$ is a condition.

Otherwise $M \cap \beta^* = M \cap \xi$ and 
$M \setminus \beta^*$ is nonempty. 
So $M \cap \beta^*$ is in $Z$. 
By the choice of $v$, 
there is $M' \in A_w$ such that $M \cap \beta^*$ is a proper initial segment of 
$M'$ and $\sup(M \cap \beta^*) = \xi \notin M'$. 
But then $M' \setminus \xi$ is nonempty and $\sup(M' \cap \xi) = \xi$. 
Since $w$ is a condition, $\xi$ must be in $M'$, which is a contradiction.

\bigskip

(6) Suppose that $M \in A_w$, $\alpha \in M$, and 
$\langle \alpha,\gamma,\beta \rangle \in x_s$. 
Then $M \in N^*$. 
Since $\alpha \in M$, $\alpha < \beta^*$, so 
$\gamma < \beta < \alpha < \beta^*$. 
Hence 
$$
\langle \alpha,\gamma,\beta \rangle \in x_s \cap Sk(\beta^*) = x_u 
\subseteq x_v \subseteq x_w.
$$
So (6) holds for $M$ and $\langle \alpha,\gamma,\beta \rangle$, because 
$w$ is a condition.

Now assume that $M \in A_s$, $\alpha \in M$, 
and $\langle \alpha,\gamma,\beta \rangle \in x_w$. 
Then $\alpha < \beta^*$. 
So $\alpha \in M \cap \beta^*$. 
Suppose that $M \cap [\gamma,\beta] \ne \emptyset$. 
Then $(M \cap \beta^*) \cap [\gamma,\beta] \ne \emptyset$. 
Since $M \cap \beta^* \in A_w$, $\gamma, \beta$ are in 
$M \cap \beta^*$, and hence in $M$, since $w$ is a condition. 

Now suppose that $M \cap [\gamma,\beta] = \emptyset$. 
Then $(M \cap \beta^*)  \cap [\gamma,\beta] = \emptyset$. 
Therefore $\sup((M \cap \beta^*) \cap \alpha) < \gamma$. 
But $(M \cap \beta^*) \cap \alpha = M \cap \alpha$. 
So $\sup(M \cap \alpha) < \gamma$.
\end{proof}

It remains to prove that $\p$ forces that $\dot S$ does not reflect. 
Towards that goal, let us first analyze the limit points of the sets 
$\dot c_\alpha$, for $\alpha \in \Gamma$.

\begin{lemma}
Let $\alpha$ be in $\Gamma$ and let $\xi < \alpha$. 
If $p$ forces that $\xi$ is a limit point of $\dot c_\alpha$, 
then there is some $M \in A_p$ such that 
$\sup(M \cap \xi) = \xi$ and $\alpha = \min(M \setminus \xi)$.
\end{lemma}

\begin{proof}
Suppose for a contradiction that 
$p$ forces that $\xi$ is a limit point of $\dot c_\alpha$, 
but there is no $M \in A_p$ as described. 
Note that for all $\langle \alpha,\gamma,\beta \rangle \in x_p$, if 
$\gamma < \xi$ then $\beta < \xi$, since otherwise $p$ would force 
that $\xi$ is not a limit point of $\dot c_\alpha$. 

Without loss of generality, we may assume that there exists 
$M \in A_p$ such that $\alpha$ and $\xi$ are in $M$. 
For otherwise we can easily extend $p$ by adding such a set $M$.

We claim that there is no $M \in A_p$ such that $\alpha \in M$, 
$\sup(M \cap \xi) < \xi$, and $M \cap [\xi,\alpha) \ne \emptyset$. 
For suppose that there was such an $M$ in $A_p$. 
Since $p$ forces that $\xi$ is a limit point of $\dot c_\alpha$, we can find 
$q \le p$ such that $\langle \alpha,\gamma,\beta \rangle \in x_q$ for some 
$\gamma, \beta < \xi$ where $\gamma > \sup(M \cap \xi)$. 
But then $M \cap [\gamma,\beta] = \emptyset$ and 
$\gamma < \sup(M \cap \alpha)$, contradicting property (6) in 
the definition of $\p$ for $q$ being a condition. 
So if $M \in A_p$, $\alpha \in M$, and $\sup(M \cap \xi) < \xi$, then 
$\sup(M \cap \xi) = \sup(M \cap \alpha)$.

We can now conclude that $\xi$ has cofinality $\omega$. 
For by assumption there exists $M \in A_p$ such that 
$\alpha$ and $\xi$ are in $M$. 
If $\cf(\xi) = \omega_1$, then $M \in A_p$, $\alpha \in M$, 
$\sup(M \cap \xi) < \xi$ since $M$ is countable, and 
$M \cap [\xi,\alpha) \ne \emptyset$ since $\xi \in M$, which 
contradicts the claim.

Define sets $A_0$, $A_1$, and $A_2$ by 
$$
A_0 := \{ M \in A_p : \alpha \notin M \},
$$
$$
A_1 := \{ M \in A_p : \alpha \in M, 
\ \sup(M \cap \alpha) < \xi \},
$$
$$
A_2 := \{ M \in A_p : \alpha \in M, \ \sup(M \cap \xi) = \xi \}.
$$
By the claim, $A_p = A_0 \cup A_1 \cup A_2$. 
By our assumption for a contradiction, 
if $M \in A_2$ then $M \cap [\xi,\alpha) \ne \emptyset$. 

Note that if $M, N \in A_1 \cup A_2$, then $\alpha \in M \cap N$, 
which implies that $\alpha < \beta_{M,N}$ by Proposition 2.6. 
In particular, if $M \in A_1$ and $N \in A_2$, then 
$M \cap \alpha \in Sk(N)$. 
For in that case 
$$
\sup(M \cap \alpha) < \xi \le \sup(N \cap \alpha) < \alpha < \beta_{M,N},
$$
which implies that $M \cap \beta_{M,N} \in Sk(N)$, since the other two 
types of comparison between $M$ and $N$ are clearly impossible.

Observe that $A_2$ is nonempty. 
For by assumption there is $M \in A_p$ such that $\alpha$ and $\xi$ 
are in $M$. 
But $\xi$ has countable cofinality, so by elementarity 
$\sup(M \cap \xi) = \xi$.

Let $M$ be $\in$-minimal in $A_2$. 
Let $\alpha^* = \min(M \setminus \xi)$. 
Then $\xi \le \alpha^* < \alpha$. 
Fix $\gamma < \xi$ in $M$ which is large enough so that 
for all $N \in A_1$, 
$\sup(N \cap \alpha) < \gamma$, and 
for all $\langle \alpha,\zeta,\beta \rangle \in x_p$, if $\zeta < \xi$ then 
$\zeta, \beta < \gamma$. 
Now define $q$ by 
$$
q := (a_p, x_p \cup \{ \langle \alpha,\gamma,\alpha^* \rangle \}, A_p).
$$
We will prove that $q$ is a condition. 
Then clearly $q$ forces that $\xi$ is not a limit point of $\dot c_\alpha$, 
and $q \le p$, which is a contradiction.

Requirements (1), (2), (3), and (5) in the definition of $\p$ 
are immediate. 
For (4), consider $\langle \alpha,\gamma',\beta' \rangle \in x_p$. 
If $\gamma' < \xi$, then by the choice of $\gamma$, 
we have that $\gamma', \beta' < \gamma$. 
So $[\gamma',\beta') \cap [\gamma,\alpha^*) = \emptyset$.

Suppose that $\gamma' \ge \xi$. 
If $M \cap [\gamma',\beta'] \ne \emptyset$, then 
$\gamma',\beta' \in M$. 
Hence $\gamma' \ge \alpha^*$ by the minimality of $\alpha^*$. 
Therefore $[\gamma,\alpha^*) \cap [\gamma',\beta') = \emptyset$.
On the other hand, if $M \cap [\gamma',\beta'] = \emptyset$, then 
since $p$ is a condition, 
$$
\alpha^* < \sup(M \cap \alpha) < \gamma'.
$$
So again $[\gamma,\alpha^*) \cap [\gamma',\beta') = \emptyset$.

For (6), suppose that $N \in A_p$ and $\alpha \in N$. 
Then $N$ cannot be in $A_0$. 
If $N \in A_1$, then $\sup(N \cap \alpha) < \gamma$ by the choice 
of $\gamma$, 
so $N \cap [\gamma,\alpha^*] = \emptyset$ and 
$\sup(N \cap \alpha) < \gamma$ as required. 

Suppose that $N \in A_2$. 
Then by the $\in$-minimality of $M$, 
either $M \cap \beta_{M,N}$ equals $N \cap \beta_{M,N}$ 
or is in $Sk(N)$. 
In either case, $M \cap \beta_{M,N} \subseteq N$. 
Since $\alpha \in M \cap N$, $\alpha < \beta_{M,N}$. 
So $\gamma$ and $\alpha^*$ are in $M \cap \beta_{M,N}$, and 
hence in $N$.
\end{proof}

\begin{proposition}
The forcing poset $\p$ forces that 
$\dot S \cap \alpha$ is nonstationary in $\alpha$, for all 
$\alpha \in \Gamma$.
\end{proposition}

\begin{proof}
Fix $\alpha \in \Gamma$. 
First let us see that $\p$ forces that $\dot c_\alpha$ is unbounded in $\alpha$. 
Let $p \in \p$ and consider $\zeta < \alpha$. 
Since $\alpha$ has cofinality $\omega_1$, we can find 
$\gamma < \alpha$ such that (1) $\zeta < \gamma$, 
(2) $\sup(M \cap \alpha) < \gamma$ for 
all $M \in A_p$, and (3) $\gamma',\beta' < \gamma$ whenever 
$\langle \alpha,\gamma',\beta' \rangle$ is in $x_p$. 
Define $q$ by 
$$
q := (a_p,x_p \cup \{ \langle \alpha, \gamma, \gamma + 1 \rangle \},A_p).
$$
It is easy to check that $q$ is a condition, and clearly $q$ forces 
that $\dot c_\alpha$ contains a point above $\zeta$.

Now suppose that $p$ forces that $\xi$ is a limit point of $\dot c_\alpha$. 
We will prove that $p$ forces that $\xi$ is not in $\dot S$. 
This argument shows that $\p$ forces that $\dot S$ is disjoint from the 
club of limit points of $\dot c_\alpha$, and hence is nonstationary 
in $\alpha$.

Suppose for a contradiction that 
there is $q \le p$ such that $q$ forces that $\xi$ is in $\dot S$. 
Then $q$ forces that there is a condition $\dot u$ in $\dot G$ such that 
$\xi \in \dot a_u$. 
Fix $s$ and $u$ such that $s \le q$ and $s$ forces that $\dot u$ is equal to $u$. 
Then $\xi$ is in $a_u$. 
Since $s$ forces that $u$ is in $\dot G$, $s$ and $u$ are compatible. 
Fix $t \le s, u$. 
Then $\xi \in a_u \subseteq a_t$. 
So $\xi \in a_t$ and $t \le p$.

Since $t$ forces that $\xi$ is a limit point of $\dot c_\alpha$, 
by Lemma 6.4 there is some $M \in A_t$ such that 
$\sup(M \cap \xi) = \xi$ and $\alpha = \min(M \setminus \xi)$. 
So we have that $\xi \in a_t$, $M \in A_t$, $\sup(M \cap \xi) = \xi$, 
and $M \setminus \xi \ne \emptyset$. 
By (5) in the definition of $\p$, $\xi$ must be in $M$. 
But $\alpha = \min(M \setminus \xi)$ implies that $\xi$ is not in $M$, 
and we have a contradiction.
\end{proof}

\bigskip

Note that in the case $\Gamma = \Lambda$, $\p$ forces that 
$\dot S \cap C$ does not reflect to any ordinal in $\omega_2 \cap \cof(\omega_1)$, 
since any such reflection point would be in $\Lambda$ since it is a limit point of 
$C$ with cofinality $\omega_1$.

\section{Adding a Kurepa Tree}

In our last application of the paper, we define a forcing poset which adds an 
$\omega_1$-Kurepa tree with finite conditions. 

Recall that an $\omega_1$-Kurepa tree is a tree with height $\omega_1$, 
all of whose levels are countable, which has more than $\omega_1$ many 
branches of length $\omega_1$. 
Such a tree can be forced using classical methods.

The conditions in our forcing poset for adding an $\omega_1$-Kurepa tree will 
include a finite tree on $\omega_1$. 
We begin by reviewing the relevant ideas and notation about finite trees, and 
prove some basic lemmas which will be useful when analyzing the 
forcing poset.

\begin{definition}
By a \emph{finite tree on $\omega_1$} we mean a pair 
$T = (|T|,<_T)$ satisfying:
\begin{enumerate}
\item $|T|$ is a finite subset of $\omega_1$;
\item $<_T$ is an irreflexive, transitive relation on $|T|$;
\item if $a, b <_T c$, then either $a = b$, $a <_T b$, or $b <_T a$;
\item $a <_T b$ implies that $a < b$.
\end{enumerate}
Given finite trees $T$ and $U$ on $\omega_1$, 
we say that $U$ \emph{end-extends} $T$ if 
$|T| \subseteq |U|$ and 
$<_U \cap \ (|T| \times |T|) = \ <_T$.
\end{definition}

Given a finite tree $T$ on $\omega_1$ and an ordinal $\alpha < \omega_1$, let 
$$
T \restriction \alpha = (|T| \cap \alpha, \ <_T \cap \ (\alpha \times \alpha)),
$$
$$
T \setminus \alpha = (|T| \setminus \alpha, 
\ <_T \cap \ ((\omega_1 \setminus \alpha) \times (\omega_1 \setminus \alpha))).
$$
Note that $T \restriction \alpha$ and $T \setminus \alpha$ are themselves 
finite trees on $\omega_1$.

\begin{definition}
Suppose that $S$ and $T$ are finite trees on $\omega_1$ and 
$\alpha < \omega_1$. 
Assume that $|T| \cap \alpha = \emptyset$ and $|S| \subseteq \alpha$. 
Let $X$ be any set of minimal nodes of $T$ and let $g : X \to |S|$ be any function. 

Define $S \oplus_{X,g} T$ as the pair $(U,<_U)$, where 
$$
|U| = |S| \cup |T|,
$$
and $x <_U y$ if either 
$x <_T y$, $x <_S y$, or there is $a \in X$ such that 
$x \le_S g(a)$ and $a \le_T y$.
\end{definition}

The purpose of this definition is to amalgamate the trees $S$ and $T$ 
in such a way that for all $a \in X$, $a$ is the immediate successor 
of $g(a)$.

\begin{lemma}
Let $S$, $T$, $\alpha$, $X$, and $g$ be as in Definition 7.2. 
Then $S \oplus_{X,g} T$ is a finite tree on $\omega_1$ which 
end-extends $S$ and $T$. 
Moreover, the maximal nodes of $S \oplus_{X,g} T$ are the 
maximal nodes of $T$ together with the maximal nodes of $S$ 
which are not in the range of $g$.
\end{lemma}

\begin{proof}
The proof is straightforward.
\end{proof}

The next lemma will be useful for amalgamating conditions in our forcing 
poset for adding a Kurepa tree.

\begin{lemma}
Let $T$ be a finite tree on $\omega_1$ and let $\alpha < \omega_1$. 
Suppose that $S$ is an end-extension of $T \restriction \alpha$ such that 
$|S| \subseteq \alpha$. 
Let $X$ be a set of minimal nodes of $T \setminus \alpha$, which includes all minimal 
nodes of $T \setminus \alpha$ which are not minimal in $T$. 
If $a \in X$ is not minimal in $T$, let $a^*$ be the immediate predecessor of $a$ in $T$. 

Let $g : X \to |S|$ be a function satisfying that for all $a \in X$:
\begin{enumerate}
\item if $a$ is not minimal 
in $T$, then $a^* \le_S g(a)$, and 
$\{ t \in |T| : a^* <_S t \le_S g(a) \} = \emptyset$;
\item if $a$ is minimal in $T$, 
then $\{ t \in |T| : t \le_S g(a) \} = \emptyset$.
\end{enumerate}

Let $U := S \oplus_{X,g} (T \setminus \alpha)$. 
Then $U$ is a finite tree on $\omega_1$ which 
end-extends $S$ and $T$. 
Moreover, the maximal nodes of $U$ are the 
maximal nodes of $T \setminus \alpha$ 
together with the maximal nodes of $S$ which are not in the range of $g$.
\end{lemma}

\begin{proof}
By Lemma 7.3, $U$ is a finite tree on $\omega_1$ 
which end-extends $T \setminus \alpha$ and $S$, and the maximal nodes 
of $U$ are the maximal nodes of $T \setminus \alpha$ together with the 
maximal nodes of $S$ which are not in the range of $g$. 
It remains to show that $U$ end-extends $T$.

Suppose that $x <_U y$, where $x, y \in |T|$. 
We will show that $x <_T y$. 
If $x$ and $y$ are below $\alpha$, then $x <_S y$, since $U$ end-extends $S$ 
and $|T \restriction \alpha| \subseteq |S|$. 
Since $S$ end-extends $T \restriction \alpha$, it follows that 
$x <_T y$. 
If $x$ and $y$ are both at least $\alpha$, then 
$x <_T y$ since $U$ end-extends $T \setminus \alpha$. 

Assume that $x < \alpha \le y$. 
Then by definition, 
$x \le_S g(a)$ and $a \le_T y$ for some $a \in X$. 
Now $a$ cannot be minimal in $T$, because otherwise by 
assumption (2), 
$\{ t \in |T| : t \le_S g(a) \} = \emptyset$, contradicting the choice of $x$. 
So by assumption (1), 
$x$ and $a^*$ are both below $g(a)$ in $S$ and 
hence are comparable. 
But by assumption (1), we cannot have $a^* <_S x$, 
therefore $x \le_S a^*$. 
Since $S$ end-extends $T \restriction \alpha$ and 
$x$ and $a^*$ are in $T \restriction \alpha$, it follows that $x \le_T a^*$. 
Therefore $x \le_T a^* <_T a \le_T y$, which implies that 
$x <_T y$.
\end{proof}

Given a model $M \in \mathcal X$, let $T \restriction M$ 
denote $T \restriction (M \cap \omega_1)$ and let 
$T \setminus M$ denote $T \setminus (M \cap \omega_1)$. 
Note that if $M \in \mathcal X$ and $\beta \in \Gamma$, then 
$M \cap \omega_1 = (M \cap \beta) \cap \omega_1$, 
so $T \restriction M = T \restriction (M \cap \beta)$ and 
$T \setminus M = T \setminus (M \cap \beta)$.

\bigskip

We are now ready to define our forcing poset for adding an $\omega_1$-Kurepa tree.
While the definition of the forcing poset is fairly simple, unfortunately 
the proofs of the preservation of $\omega_1$ and $\omega_2$ are 
quite involved.

\begin{definition}
Let $\p$ be the forcing poset consisting of triples $( T, F, A )$ satisfying:
\begin{enumerate}
\item $T = ( |T|, <_T )$ is a finite tree on $\omega_1$;
\item $F$ is an injective function from the maximal nodes of $T$ into $\omega_2$;
\item $A$ is a finite adequate set;
\item if $M \in A$, $a$ and $b$ are distinct maximal nodes of $T$, and 
$F(a)$ and $F(b)$ are in $M$, 
then for any $c$ which is below both $a$ and $b$ in $T$, 
$c$ is in $M$.
\end{enumerate}
Let $( U, G, B ) \le (T, F, A )$ 
if $U$ end-extends $T$, $A \subseteq B$, 
and whenever $a$ is maximal in $T$, then there is $b$ which is 
maximal in $U$ such that 
$a \le_U b$ and $F(a) = G(b)$.
\end{definition}

If $p = (T,F,A)$, then we let $T_p := T$, 
$F_p := F$, and $A_p := A$.

Note that if $p$ is a condition, $M_1, \ldots, M_k \in A_p$, 
and $\beta_1, \ldots, \beta_k \in \Gamma$, then 
$(T_p,F_p,A_p \cup \{ M _1 \cap \beta_1, \ldots, M_k \cap \beta_k \})$ is a condition. 
For requirements (1), (2), and (3) in the definition of $\p$ 
are immediate, and (4) is preserved under 
taking initial segments of models.

\bigskip

The proofs that $\p$ preserves $\omega_1$ and $\omega_2$ will take 
some time. 
Let us temporarily assume that $\p$ preserves $\omega_1$ and $\omega_2$, 
and show how the forcing poset $\p$ adds an $\omega_1$-Kurepa tree.
Note that since $\p$ has size $\omega_2$, it preserves cardinals larger 
than $\omega_2$ as well.

\bigskip

Observe that for any ordinal $\alpha < \omega_1$, 
there are densely many $q$ with 
$\alpha \in |T_q|$. 
Indeed, given a condition $p$, if $\alpha$ is not already in $|T_p|$, then let 
$$
T_q = (|T_p| \cup \{ \alpha \},<_{T_p}),
$$
and extend $F_p$ to $F_q$ by letting 
$F_q(\alpha)$ be any value not in the range of $F_p$. 
Then easily $q = (T_q,F_q,A_p)$ is a condition below $p$.

\bigskip

Let $\dot R$ be a $\p$-name such that $\p$ forces that 
$\dot R$ is the set of pairs $( \alpha, \beta )$ 
for which there exists $p \in \dot G$ such that 
$\alpha <_{T_p} \beta$. 
Let $\dot T$ be a $\p$-name for the pair $(\omega_1,\dot R)$. 
It is straightforward to prove that $\p$ forces that 
$\dot T$ is a tree which end-extends $T_p$ for all $p \in \dot G$.

The next two lemmas will establish that $\p$ forces that $\dot T$ is 
an $\omega_1$-Kurepa tree.

\begin{lemma}
The forcing poset $\p$ forces that each level of $\dot T$ is countable.
\end{lemma}

\begin{proof}
Suppose for a contradiction that there is a condition $p$ 
and an ordinal $\alpha < \omega_1$ 
such that $p$ forces that $\alpha$ is the least ordinal such that 
level $\alpha$ of $\dot T$ is uncountable. 
Then $p$ forces that the set of nodes which belong to a level less than $\alpha$ 
is countable. 
As a result, it is easy to see that there exists 
$q$, $\gamma$, and $b$ satisfying:
\begin{enumerate}
\item $q \le p$;
\item $b \in T_q$;
\item $b \ge \gamma + \omega$;
\item $q$ forces that $b$ is on level $\alpha$ in $\dot T$;
\item $q$ forces that any node of $\dot T$ on a level less than $\alpha$ 
is less than $\gamma$.
\end{enumerate}
Note that for any $\xi$ with $\gamma \le \xi < b$, $q$ forces that 
$\xi$ is not below $b$ in $\dot T$. 
For otherwise as $b$ is on level $\alpha$ by (4), $\xi$ would be on a level less 
than $\alpha$, and hence below $\gamma$ by (5).

Choose an ordinal $a$ such that $\gamma \le a < b$ and $a$ is 
different from any ordinal in $|T_q|$, which is possible since $|T_q|$ is finite. 
Define $T_r$ by letting $|T_r| = |T_q| \cup \{ a \}$, 
and letting $x <_{T_r} y$ if either:
\begin{enumerate}
\item $x <_{T_q} y$, or
\item $x <_{T_q} b$ and $y = a$, or
\item $x = a$ and $b \le_{T_q} y$.
\end{enumerate} 
In other words, we add $a$ so that it is an immediate predecessor of $b$. 
Easily $T_r$ is a tree which end-extends $T_q$. 
Also $T_q$ and $T_r$ have the same maximal nodes.

Let $r = (T_r,F_q,A_q)$. 
We claim that $r$ is a condition. 
Requirements (1), (2), and (3) in the definition of $\p$ are immediate. 
For (4), let $M \in A_r$ and suppose that $d$ and $e$ are distinct maximal 
nodes of $T_r$, $F_r(d)$ and $F_r(e)$ are in $M$, and 
$c <_{T_r} d, e$. 
Note that $d, e \in |T_q|$, since $T_q$ and $T_r$ have the 
same maximal nodes.
 
If $c \in |T_q|$, then $c \in M$ since $q$ is a condition. 
Otherwise $c = a$. 
Since $b$ is the unique immediate successor of $a$, and $d$ and $e$ are distinct, 
we must have that $b <_{T_r} d, e$. 
But then $b \in M$ since $q$ is a condition. 
Since $a < b$, $a \in M$ because $M \cap \omega_1$ is an ordinal. 

So indeed $r$ is a condition. 
Clearly $r \le q$. 
But this is a contradiction since $a \ge \gamma$ and $r$ forces 
that $a$ is below $b$ in $\dot T$.
\end{proof}

\begin{lemma}
The forcing poset $\p$ forces that $\dot T$ has $\omega_2$ many 
distinct branches.
\end{lemma}

\begin{proof}
For each $i < \omega_2$, let 
$\dot b_i$ be a name such that $\p$ forces that 
$a \in \dot b_i$ iff 
for some $p \in \dot G$, there is a maximal node $b$ of $T_p$ such that 
$a \le_{T_p} b$ and $F_p(b) = i$.
We will prove that $\p$ forces that $\langle \dot b_i : i < \omega_2 \rangle$ 
is a sequence of distinct branches of $\dot T$ each of length $\omega_1$. 

Let $G$ be a generic filter on $\p$, 
and let $T := \dot T^G$ and $b_i := \dot b_i^G$. 
To show that $b_i$ is a chain, suppose that $\alpha$ and $\beta$ are in $b_i$, 
and we will 
show that they are comparable in $T$.

Fix $p$ and $q$ in $G$ such that there are maximal nodes $b$ and $c$ 
of $T_p$ and $T_q$ above $\alpha$ and $\beta$ respectively such that 
$F_p(b) = F_q(c) = i$. 
Fix $r$ in $G$ below $p$ and $q$. 
By the definition of the ordering on $\p$, 
there are maximal nodes $b'$ and $c'$ above $b$ and $c$ respectively in $T_r$ 
such that $F_r(b') = F_p(b) = i$ and $F_r(c') = F_q(c) = i$. 
Since $F_r$ is injective, $b' = c'$. 
Hence $\alpha$ and $\beta$ are below the same node in $T_r$, and therefore 
since $T_r$ is a tree, they are comparable in $T_r$, and hence in $T$.

To show that $b_i$ has length $\omega_1$, 
it is enough to show that there are 
cofinally many $\alpha$ in $\omega_1$ which are in $b_i$. 
By a density argument, it suffices to show that whenever $p \in \p$ 
and $\gamma < \omega_1$, there is $q \le p$ 
and $a \ge \gamma$ such that $a$ is a maximal node of 
$T_q$ and $F_q(a) = i$.

Fix $\alpha$ such that $\gamma < \alpha < \omega_1$ and $\alpha$ is larger 
than all the ordinals in $|T_p|$. 
If there does not exist a maximal node $b$ in $T_p$ such that $F_p(b) = i$, 
then let $T_q = (|T_p| \cup \{ \alpha \}, <_{T_p})$ and 
$F_q = F_p \cup \{ (\alpha,i) \}$. 
Then $q = (T_q,F_q,A_p)$ is as desired.

Now suppose that there is a maximal node $b$ in $T_p$ such that 
$F_p(b) = i$. 
Then define $T_q$ by adding $\alpha$ as an immediate successor of $b$. 
Extend $F_p$ to $F_q$ by letting $F_q(\alpha) = i$. 
It is easy to check that $q = (T_q,F_q,A_p)$ is a condition, and clearly 
$q$ is as desired.

Finally, we show that if $i \ne j$ then $b_i$ and $b_j$ are distinct. 
The argument in the previous two paragraphs shows that given a condition $p$, 
we can extend $p$ to $q$ so that 
there are maximal nodes $a$ and $b$ of 
$T_q$ such that $F_q(a) = i$ and $F_q(b) = j$. 
Then $q$ forces that $a \in \dot b_i$ and $b \in \dot b_j$.

We claim that $q$ forces that $a \notin \dot b_j$. 
This implies that $q$ forces that $\dot b_i \ne \dot b_j$, which finishes the proof. 
Otherwise there is $r \le q$ and a maximal node $c$ of $T_r$ such that 
$a \le_{T_r} c$ and $F_r(c) = j$. 
Since $r \le q$, there is a maximal node $d$ of $T_r$ such that 
$b \le_{T_r} d$ and $F_r(d) = F_q(b) = j$. 
As $F_r$ is injective, $c = d$. 
But then $a$ and $b$ are both below $c$ in $T_r$, which implies 
that they are comparable 
in $T_r$. 
Hence they are comparable in $T_q$ since $T_r$ end-extends $T_q$. 
This is a contradiction since $a$ and $b$ are distinct maximal nodes of $T_q$.
\end{proof}

We now turn to showing that $\p$ preserves $\omega_1$ and $\omega_2$. 
For the preservation of $\omega_1$, it will be useful to first describe 
a dense subset of conditions which will help in the amalgamation argument.

\begin{lemma}
Let $p$ be a condition and let $N \in A_p$. 
Then there exists $r \le p$ satisfying:
\begin{enumerate}
\item $T_r$ has no maximal nodes which are less than $N \cap \omega_1$; 
\item the function which sends a minimal node of 
$T_r \setminus N$ to its immediate predecessor in $T_r$, if it exists, 
is injective and its range is an antichain.
\end{enumerate}
\end{lemma}

\begin{proof}
Let $c_1, \ldots, c_m$ denote the maximal nodes of $T_p$ which 
are below $N \cap \omega_1$. 
Choose distinct ordinals $\beta_1, \ldots, \beta_m$ in $\omega_1$ which are 
larger than $N \cap \omega_1$ and larger than all ordinals 
appearing in $T_p$. 

We define $q = (T_q,F_q,A_q)$ as follows. 
Extend $T_p$ to $T_q$ by placing $\beta_i$ as the immediate 
successor of $c_i$, for each $i = 1, \ldots, m$. 
Let $F_q(\beta_i) := F_p(c_i)$, for each $i = 1, \ldots, m$. 
If $a$ is a maximal node of $T_q$ different from the $\beta_i$'s, 
then $a$ is a maximal node of $T_p$ and $a \ge N \cap \omega_1$; 
in that case, let $F_q(a) := F_p(a)$.
Let $A_q := A_p$.
The proof that $q$ is a condition below $p$ is straightforward, 
and $q$ clearly satisfies (1).

We further extend $q$ to $r$ which satisfies both (1) and (2). 
Let $X$ be the set of minimal nodes of $T_q \setminus N$ which are not 
minimal in $T_q$. 
For each $a \in X$, let $a'$ be the immediate predecessor of $a$ in $T_q$. 
Now choose for each $a \in X$ some ordinal $g(a)$ in $N$ larger than $a'$ and 
different from the ordinals in $T_q$. 
We also choose the values for $g$ so that $g$ is injective. 
This is possible since $|T_q|$ is finite. 
Let $S$ be the tree 
obtained from $T_q \restriction N$ by adding $g(a)$ above $a'$ 
for each $a \in X$. 

Now clearly $T_q$, $N \cap \omega_1$, $S$, $X$, and $g$ satisfy the 
assumptions of Lemma 7.4. 
So we can define $T_r$ as the tree 
$S \oplus_{X,g} (T_q \setminus N)$, which amalgamates 
$S$ and $T_q$. 
Since $T_q$ has no maximal nodes below $N \cap \omega_1$, 
every maximal node of $S$ is in the range of $g$. 
By Lemma 7.4, it follows that $T_q$ and $T_r$ have the same maximal nodes. 
So we can define $F_r := F_q$. 
Let $A_r := A_q$.

Since $T_q$ and $T_r$ have the same maximal nodes 
and $T_q$ satisfies property (1), also $T_r$ satisfies property (1).
For property (2), if $a$ is a minimal node of $T_r \setminus N$ which is 
not minimal in $T_r$, then $a \in X$, and 
the immediate predecessor of $a$ in $T_r$ is $g(a)$. 
Since $g(a)$ and $g(b)$ are distinct and incomparable in $T_r$, for any distinct 
$a$ and $b$ in $X$, it follows that $T_r$ satisfies property (2).

It remains to show that $r = (T_r,F_r,A_r)$ is a condition. 
Requirements (1), (2), and (3) in the definition of $\p$ are immediate. 
For (4), let $M \in A_r$, and suppose that $c <_{T_r} a, b$, where 
$a, b$ are maximal in $T_r$ and $F_r(a), F_r(b) \in M$. 
We will show that $c$ is in $M$. 
Since $T_q$ and $T_r$ have the same maximal nodes, $a$ and $b$ are in $T_q$.

First, assume that $c$ is in $|T_q|$. 
Then since $T_r$ end-extends $T_q$, $c <_{T_q} a, b$. 
Since $M \in A_r$ and $A_q = A_r$, we have that $M \in A_q$. 
Also $F_r = F_q$, so $F_q(a), F_q(b) \in M$. 
Since $q$ is a condition, it follows that $c \in M$.

Secondly, assume that $c$ is not in $|T_q|$. 
Then $c = g(x)$, for some $x \in X$. 
By the definition of $T_r$, we have that $x \le_{T_q} a, b$. 
Note that it is impossible that $x$ is equal to $a$ or $b$, 
since otherwise $a$ and $b$ would be comparable, which contradicts 
that $a$ and $b$ are distinct maximal nodes of $T_q$. 
So in fact $x <_{T_q} a, b$. 
Therefore by the previous paragraph, $x \in M$. 
Since $c < x$ and $x \in M \cap \omega_1$, it follows that $c \in M$. 
\end{proof}

\begin{proposition}
The forcing poset $\p$ is strongly proper on a stationary set.
\end{proposition}

\begin{proof}
Fix $\theta > \omega_2$ regular. 
Let $N^*$ be a countable elementary substructure of $H(\theta)$ satisfying 
that $\p, \pi, \mathcal X \in N^*$ 
and $N := N^* \cap \omega_2 \in \mathcal X$. 
Note that since $\mathcal X$ is stationary, there are stationarily many 
such $N^*$ in $P_{\omega_1}(H(\theta))$. 
To prove the proposition, it suffices to show that every condition in 
$N^* \cap \p$ has a strongly $(N^*,\p)$-generic extension. 

Observe that since $\pi \in N^*$ and $\pi : \omega_2 \to H(\omega_2)$ 
is a bijection, by elementarity we have that 
$$
N^* \cap H(\omega_2) = \pi[N^* \cap \omega_2] = \pi[N] = Sk(N),
$$
where the last equality holds by Lemma 1.3 
and the fact that 
$N \in \mathcal X$ implies that $Sk(N) \cap \omega_2 = N$. 
In particular, $N^* \cap \p \subseteq Sk(N)$.

Fix $p \in N^* \cap \p$. 
Then as just noted, $p \in Sk(N)$. 
Define 
$$
q := (T_p,F_p,A_p \cup \{ N \}).
$$
Then easily $q$ is a condition, and $q \le p$. 
We will show that $q$ is strongly $(N^*,\p)$-generic. 
Fix a set $D$ which is a dense subset of $N^* \cap \p$, 
and we will show that $D$ is predense below $q$.

\bigskip

Let $r \le q$ be given. 
We will find a condition in $D$ which is compatible with $r$. 
Applying Lemma 7.8, 
we can fix $r' \le r$ satisfying that $T_{r'}$ has no maximal nodes 
below $N \cap \omega_1$, and the function which sends a minimal node of 
$T_{r'} \setminus N$ to its immediate predecessor, if it exists, is injective and its 
range is an antichain.

We extend $r'$ to prepare for intersecting with $N^*$. 
Define $s$ by letting 
$T_s := T_{r'}$, $F_s := F_{r'}$, and 
$$
A_s := A_{r'} \cup \{ M \cap \beta_{M,N} : M \in A_{r'}, \ 
M \cap \beta_{M,N} \in Sk(N) \}.
$$
By the comments after Definition 7.5, $s$ is a condition, and obviously 
$s \le r'$. 
Moreover, it is easy to see that 
$s$ satisfies properties (1) and (2) of Lemma 7.8, since $r'$ does. 
As $s \le r$, we will be done if we can find a condition in $D$ which is 
compatible with $s$.

Let $M_1, \ldots, M_k$ enumerate the sets $M$ in $A_s$ such that 
$M \cap \beta_{M,N} \in Sk(N)$ and $M \setminus \beta_{M,N} \ne \emptyset$.

To find a condition in $D$ which is compatible with $s$, we first 
need to find a condition 
in $N^*$ which reflects some information about $s$. 

\bigskip

\noindent \textbf{Main Claim.} 
There exists a condition $v \in N^*$ satisfying:
\begin{enumerate}
\item there is an isomorphism $\sigma : T_s \to T_v$ which is the 
identity on $T_s \restriction N$;
\item for all $y \in T_s \setminus N$ and $i = 1, \ldots, k$, 
$\sigma(y) > M_i \cap \omega_1$;
\item if $x$ is maximal in $T_s$ and $F_s(x) \in N$, then 
$F_v(\sigma(x)) = F_s(x)$;
\item there are $L_1, \ldots, L_k$ in $A_v$ 
such that $L_i$ end-extends $M_i \cap \beta_{M_i,N}$ 
for each $i = 1, \ldots, k$;
\item for each maximal node $a$ of $T_s$ and each $i = 1, \ldots, k$, 
if $F_s(a) \in M_i \setminus N$, then 
$F_v(\sigma(a)) \in L_i \setminus (M_i \cap \beta_{M_i,N})$;
\item $A_s \cap N^* \subseteq A_v$.
\end{enumerate}

\bigskip

We prove the claim. 
Let $\alpha_1, \ldots, \alpha_m$ and $\beta_1, \ldots, \beta_n$ 
list the elements of $|T_s| \cap N$ and $|T_s| \setminus N$ respectively in ordinal 
increasing order. 
Define sets $P_1, \ldots, P_k$ which are subsets of $\{ 1, \ldots, n \}$ by letting 
$j \in P_i$ if $\beta_j$ is maximal in $T_s$ and $F_s(\beta_j) \in M_i \setminus N$. 
Let $S$ be the set of $j \in \{ 1, \ldots, n \}$ such that 
$\beta_j$ is maximal in $T_s$ and $F_s(\beta_j) \in N$. 
For each $j \in S$ let $\xi_j := F_s(\beta_j)$, which by definition is a member of $N$. 
Let $\Sigma$ be an integer which codes the isomorphism type of the finite structure 
$$
(|T_s|, <_{T_s}, \alpha_1, \ldots, \alpha_m, \beta_1, \ldots, \beta_n).
$$

The objects $s$, $\beta_1, \ldots, \beta_n$, and $M_1, \ldots, M_k$ witness 
that there is $v \in \p$, $\gamma_1, \ldots, \gamma_n$, and 
$L_1, \ldots, L_k$ satisfying:
\begin{enumerate}
\item[(i)] $\gamma_1, \ldots, \gamma_n$ is an increasing sequence of ordinals 
larger than $\alpha_1, \ldots, \alpha_m$ and larger 
than $(M_1 \cap \beta_{M_1,N}) \cap \omega_1, 
\ldots, (M_k \cap \beta_{M_k,N}) \cap \omega_1$ 
such that the structure 
$$
(|T_v|, <_{T_v}, \alpha_1, \ldots, \alpha_m, \gamma_1, \ldots, \gamma_n)
$$
has isomorphism type $\Sigma$;
\item[(ii)] $L_1, \ldots, L_k$ are in $A_v$ and for each $i = 1, \ldots, k$, 
$L_i$ end-extends $M_i \cap \beta_{M_i,N}$;
\item[(iii)] for each $i =1, \ldots, k$, $j \in P_i$ iff $\gamma_j$ is maximal in $T_v$ 
and $F_v(\gamma_j) \in L_i \setminus (M_i \cap \beta_{M_i,N})$;
\item[(iv)] for all $j \in S$, $\gamma_j$ is maximal in $T_v$ and 
$F_v(\gamma_j) = \xi_j$;
\item[(v)] $A_s \cap N^* \subseteq A_v$.
\end{enumerate}
Now the parameters which appear in the above statement, namely, 
$\p$, $\alpha_1, \ldots, \alpha_m$, $M_1 \cap \beta_{M_1,N}, \ldots, 
M_k \cap \beta_{M_k,N}$, $\omega_1$, $\Sigma$, $P_1, \ldots, P_k$, 
$S$, $\langle \xi_j : j \in S \rangle$, and $A_s \cap N^*$, are all members of $N^*$. 
So by the elementarity of $N^*$, we can fix 
$v \in \p$, $\gamma_1, \ldots, \gamma_n$, and 
$L_1, \ldots, L_k$ which are members of $N^*$ and 
satisfy the same statement.

\bigskip

Let us show that $v$ is as required. 
We know that $v$ is in $N^* \cap \p$. 
Requirement (4) in the claim follows from (ii), 
and (6) follows from (v).
 
Define $\sigma : T_s \to T_v$ by letting $\sigma(\alpha_i) := \alpha_i$ for 
$i = 1, \ldots, m$, and $\sigma(\beta_j) := \gamma_j$ for $j = 1, \ldots, n$. 
Then by the choice of $\Sigma$, $\sigma$ is an isomorphism, and 
$\sigma$ is the identity on $T \restriction N$. 
Thus (1) holds. 
(2) follows from (i). 
It remains to prove (3) and (5).

For (3), suppose that $x$ is maximal in $T_s$ and $F_s(x) \in N$. 
Fix $j$ such that $x = \beta_j$. 
Then $j \in S$, by the definition of $S$. 
Also $\sigma(x) = \gamma_j$. 
By (iv), 
$$
F_v(\sigma(x)) = F_v(\gamma_j) = \xi_j = F_s(\beta_j) = F_s(x).
$$

For (5), let $a$ be a maximal node of $T_s$, 
and suppose that $F_s(a) \in M_i \setminus N$ 
for some $i = 1, \ldots, k$. 
Fix $j$ such that $a = \beta_j$. 
Then $j \in P_i$, by the definition of $P_i$. 
So by (iii), $\gamma_j$ is maximal in $T_v$ and 
$$
F_v(\gamma_j) \in L_i \setminus (M_i \cap \beta_{M_i,N}).
$$
But $\gamma_j = \sigma(\beta_j) = \sigma(a)$. 
So 
$$
F_v(\sigma(a)) \in L_i \setminus (M_i \cap \beta_{M_i,N}).
$$
This completes the proof of the main claim.

\bigskip

Since $D$ is a dense subset of $N^* \cap \p$, we can fix $w \le v$ in $D$. 
We will show that $w$ and $s$ are compatible, which completes the proof. 
We define a condition $z = (T_z,F_z,A_z)$, and prove that 
$z \le w, s$.

\bigskip

First, let $A_z := A_s \cup A_w$. 
Note that $A_z$ is adequate by Proposition 3.9.

\bigskip

Secondly, we apply Lemma 7.4 to amalgamate the trees $T_w$ and $T_s$. 
Let $X$ be the set of all minimal nodes $a$ of $T_s \setminus N$ such that 
either $a$ is not minimal in $T_s$, or there is a maximal node $d$ with 
$a \le_{T_s} d$ and $F_s(d) \in N$. 
Note that in the second case, $d$ is unique, since otherwise by (4) in the definition 
of $\p$, $a$ would be in $N$. 
For each $a$ in $X$ which is not minimal in $T_s$, let $a^*$ be the 
immediate predecessor of $a$ in $T_s$. 
Recall that since $s$ satisfies property (2) of Lemma 7.8, 
$a^*$ and $b^*$ are distinct and 
incomparable for different $a$ and $b$.

We define an injective function 
$g : X \to |T_w|$ which will 
satisfy the assumptions of Lemma 7.4, namely, that:
\begin{enumerate}
\item[(a)] for all $a \in X$, if $a$ is not minimal 
in $T_s$, then $a^* \le_{T_w} g(a)$, and
$\{ t \in |T_s| : a^* <_{T_w} t \le_{T_w} g(a) \} = \emptyset$;
\item[(b)] if $a$ is minimal in $T_s$, 
then $\{ t \in |T_s| : t \le_{T_w} g(a) \} = \emptyset$.
\end{enumerate}
So fix $a \in X$, and we define $g(a)$.

\bigskip

\noindent Case 1: There does not exist a maximal node $d$ of $T_s$ such that 
$a \le_{T_s} d$ and $F_s(d) \in N$. 
Then by the definition of $X$, $a$ is not minimal in $T_s$. 
Let $g(a) = a^*$. 
Clearly, requirements (a) and (b) are satisfied.

\bigskip
 
\noindent Case 2: There exists a maximal node $d$ of $T_s$ such that 
$a \le_{T_s} d$ and $F_s(d) \in N$. 
Then $d$ is unique, as observed above. 
By (3) in the main claim, 
$$
F_v(\sigma(d)) = F_s(d).
$$
Since $w \le v$, by the definition of the ordering on $\p$ there is a unique 
maximal node $\sigma^+(d)$ of $T_w$ above $\sigma(d)$ 
such that 
$$
F_w(\sigma^+(d)) = F_v(\sigma(d)).
$$
Let $g(a) = \sigma^+(d)$. 
Then by the above equations, 
$$
F_s(d) = F_w(g(a)).
$$
Let us check that g(a) satisfies requirements (a) and (b).

(a) Assume that $a$ is not minimal in $T_s$. 
Then $a^* <_{T_s} a \le_{T_s} d$, so $a^* <_{T_s} d$. 
Since $a^* <_{T_s} d$ and $\sigma$ is an isomorphism which is the 
identity on $T_s \restriction N$, we have that 
$$
\sigma(a^*) = a^* <_{T_v} \sigma(d).
$$
Since $w \le v$ and $\sigma(d) \le_{T_w} \sigma^+(d) = g(a)$, 
we have that 
$$
a^* <_{T_w} \sigma(d) \le_{T_w} g(a),
$$
and hence $a^* <_{T_w} g(a)$, which proves the first part of (a).

For the second part of (a), 
suppose for a contradiction that there exists $t$ in $T_s$ such that 
$$
a^* <_{T_w} t \le_{T_w} g(a).
$$
Since $g(a) \in N \cap \omega_1$, also $t \in N$. 
As $T_w$ end-extends $T_s \restriction N$ and 
$a^*$ and $t$ are in $T_s \restriction N$, 
we have that $a^* <_{T_s} t$. 

Now 
$$
t = \sigma(t) \le_{T_w} g(a) = \sigma^+(d).
$$
Since also $\sigma(d) \le_{T_w} \sigma^+(d)$, we have that 
$\sigma(t) = t$ and $\sigma(d)$ are comparable in $T_w$, since 
$T_w$ is a tree. 
Hence they are comparable in $T_v$, since $T_w$ end-extends $T_v$ 
and $t$ and $\sigma(d)$ are in $T_v$. 
But $\sigma(d)$ is maximal in 
$T_v$, since $\sigma$ is an isomorphism. 
Therefore $\sigma(t) = t \le_{T_v} \sigma(d)$. 
It follows that $t \le_{T_s} d$, since $\sigma$ is an isomorphism. 
But $t$ is in $N$ and $d$ is not in $N$, so $t <_{T_s} d$. 

Now $a$ and $t$ are distinct nodes below $d$ in $T_s$, and 
$t < N \cap \omega_1 \le a$. 
So $t <_{T_s} a$, since $T_s$ is a tree.  
Hence we have that 
$$
a^* <_{T_s} t <_{T_s} a,
$$
which contradicts the fact that $a^*$ is the immediate predecessor of $a$ in $T_s$.

(b) Suppose that $a$ is minimal in $T_s$. 
Assume for a contradiction that $t \in |T_s|$ and 
$$
t \le_{T_w} g(a) = \sigma^+(d).
$$
Since $\sigma(d) \le_{T_w} \sigma^+(d)$, we have that 
$t$ and $\sigma(d)$ are comparable in $T_w$. 
But $t$ and $\sigma(d)$ are in $T_v$, so they are comparable in $T_v$, 
since $T_w$ end-extends $T_v$. 
As $\sigma(d)$ is maximal in $T_v$, we have that 
$$
\sigma(t) = t \le_{T_v} \sigma(d).
$$
Since $\sigma$ is an isomorphism, it follows that $t \le_{T_s} d$. 
Since also $a \le_{T_s} d$ and $t < N \cap \omega_1 \le a$, we have that 
$t <_{T_s} a$. 
This contradicts the assumption that $a$ is minimal in $T_s$.

This completes the proof that 
$g$ satisfies the assumptions of Lemma 7.4. 
It is easy to check by cases that $g$ is injective, using the fact that the map which 
sends a minimal node $a$ of $T_s \setminus N$ to its predecessor $a^*$ in 
$T_s$, if it exists, is injective.

\bigskip

Let $T_z := T_w \oplus_{X,g} (T_s \setminus N)$. 
Then by Lemma 7.4, $T_z$ end-extends $T_w$ and $T_s$. 
Moreover, the maximal nodes of $T_z$ are the maximal nodes of $T_s$ 
together with 
the maximal nodes of $T_w$ which are not in the range of $g$. 
Note that since $s$ satisfies property (1) of Lemma 7.8, 
any maximal node of $T_s$ is at least $N \cap \omega_1$, 
and so is not also a maximal node of $T_w$. 

\bigskip

Thirdly, we define the function $F_z$. 
Let $a$ be a maximal node of $T_z$. 
Then as just mentioned, 
there are two disjoint possibilities. 
First, suppose that $a$ is a maximal node of $T_s$. 
In this case, let $F_z(a) := F_s(a)$. 
Secondly, suppose that $a$ is a maximal node of $T_w$ which is 
not in the range of $g$. 
In this case, let $F_z(a) := F_w(a)$.

\bigskip

This completes the definition of $z$. 
We will be done if we can show that $z$ is a condition, and 
$z \le w, s$. 
The proof that $z$ is a condition will take some time. 
So let us temporarily assume that $z$ is a condition, and show that 
$z \le w, s$.

We already know that $T_z$ end-extends $T_w$ and $T_s$. 
Also $A_w, A_s \subseteq A_z$, by the definition of $A_z$.

To show that $z \le s$, let $c$ be maximal in $T_s$. 
Then $c \ge N \cap \omega_1$, since $s$ satisfies property (1) of Lemma 7.8. 
So $c$ is still maximal in $T_z$ and 
$F_z(c) = F_s(c)$. 
This proves that $z \le s$. 

To show that $z \le w$, let $c$ be maximal in $T_w$. 
If $c$ is still maximal in $T_z$, then 
then $F_z(c) = F_w(c)$, and we are done. 
Otherwise $c$ is in the range of $g$. 
Hence $c = g(y)$, for some minimal node $y$ of $T_s \setminus N$.

There are two possibilities, based on the case division in the definition of $g$. 
First, assume that case 1 in the definition of $g$ holds. 
Then $c = g(y) = y^*$, which is the predecessor of $y$ in $T_s$. 
Since $y^* <_{T_s} y$, it follows that 
$$
\sigma(y^*) = y^* <_{T_v} \sigma(y),
$$
since $\sigma$ is an isomorphism which 
is the identity on $T_s \restriction N$. 
As $T_w$ end-extends $T_v$, we have that 
$c = y^* <_{T_w} \sigma(y)$. 
But this contradicts the assumption that $c$ is maximal in $T_w$. 

Secondly, assume case 2 in the definition of $g$. 
Then there is a maximal node $d$ of $T_s \setminus N$ 
such that $F_s(d) \in N$, 
there is $y$ which is minimal in $T_s \setminus N$ such that 
$y \le_{T_s} d$, 
$$
c = g(y) = \sigma^+(d),
$$
and 
$$
F_w(c) = F_w(\sigma^+(d)) = F_v(\sigma(d)) = F_s(d),
$$
where the last equality holds by (3) of the main claim. 
Then $d$ is maximal in $T_z$, $c \le_{T_z} d$, and $F_w(c) = F_z(d)$, as required.
This completes the proof that $z \le w$.

\bigskip

In order to prove that $z$ is a condition, we verify requirements 
(1)--(4) in the definition of $\p$. 
(1) is clear, and for (3), we have already observed above that $A_z$ is adequate.

\bigskip

(2) Let us prove that $F_z$ is injective. 
Since $w$ and $s$ are conditions, $F_z$ is injective on the maximal nodes of $T_s$, 
and $F_z$ is injective on the maximal nodes of $T_w$ which are not in the 
range of $g$. 
So the only nontrivial case to consider is when $d$ is maximal in $T_s$ 
and $d'$ is maximal in $T_w$ but not in the range of $g$. 
Then $F_z(d) = F_s(d)$ and $F_z(d') = F_w(d')$.
We will show that $F_z(d) \ne F_z(d')$, that is, that 
$F_s(d) \ne F_w(d')$.

Since $w \in N^*$, $F_w(d') \in N$. 
So if $F_s(d) \notin N$, then $F_w(d') \ne F_s(d)$, and we are done. 
Assume that $F_s(d) \in N$. 
Let $a$ be the unique minimal node of $T_s \setminus N$ with 
$a \le_{T_s} d$. 
Since $F_s(d) \in N$, by case 2 in the definition of $g$, 
$$
g(a) = \sigma^+(d).
$$
But 
$$
F_w(g(a)) = F_w(\sigma^+(d)) = F_v(\sigma(d)),
$$
and by (3) in the main claim, 
$$
F_v(\sigma(d)) = F_s(d).
$$
So $F_w(g(a)) = F_s(d)$.

Since $d'$ is maximal in $T_z$, it is not in the range of $g$; hence 
$d' \ne g(a)$. 
Since $F_w$ is injective, $F_w(d') \ne F_w(g(a))$. 
So by the definition of $F_z$ and the fact that 
$F_w(g(a)) = F_s(d)$, we have 
$$
F_z(d') = F_w(d') \ne F_w(g(a)) = F_s(d) = F_z(d).
$$
So $F_z(d') \ne F_z(d)$, as required.

\bigskip

(4) Let $M \in A_z$, and assume that $a$ and $b$ are 
distinct maximal nodes of $T_z$ 
such that $F_z(a), F_z(b) \in M$. 
Let $c <_{T_z} a, b$. 
We will prove that $c \in M$. 

If either of $a$ or $b$ are in $M$, 
then so is $c$ because $M \cap \omega_1$ is an ordinal. 
So assume that neither $a$ nor $b$ is in $M$.

Let us first handle the case when $c$ is not in $N$. 
Then neither are $a$ and $b$, since $N \cap \omega_1$ is an ordinal 
and $c$ is less than $a$ and $b$.  
So $a, b, c$ are in $T_s$. 
If $M \in A_s$, then we are done since $s$ is a condition. 
If $M$ is not in $A_s$, then $M$ is in $A_w$ and hence in $Sk(N)$. 
Since $F_s(a)$ and $F_s(b)$ are in $M$ and $M \subseteq N$, 
$F_s(a)$ and $F_s(b)$ are in $N$. 
By requirement (4) of $s$ being a condition, 
it follows that $c \in N$, which contradicts 
our assumption that $c$ is not in $N$.

For the remainder of the proof we will assume that $c$ is in $N$. 
If $N \cap \beta_{M,N}$ is either equal to $M \cap \beta_{M,N}$ or in $Sk(M)$, 
then 
$$
c \in N \cap \omega_1 = (N \cap \beta_{M,N}) \cap \omega_1 
\subseteq M,
$$
so $c \in M$ and we are done. 
Thus for the remainder of the proof we will assume that 
$M \cap \beta_{M,N} \in Sk(N)$.

\bigskip

\noindent Case A: $F_z(a), F_z(b) \in N$. 
Then 
$$
F_z(a), F_z(b) \in M \cap N \subseteq M \cap \beta_{M,N}.
$$
Note that there are $a'$ and $b'$ maximal in $T_w$ such that 
$$
a' \le_{T_z} a, \ b' \le_{T_z} b, \ 
F_w(a') = F_z(a), \ \textrm{and} \ F_w(b') = F_z(b).
$$
Namely, if $a$ is in $N$, then let $a' := a$, and if $b$ is in $N$, then let $b' := b$. 
If $a$ is not in $N$, then let $a' := \sigma^+(a)$, and similarly with $b$. 
Then $a'$ and $b'$ are as desired.

Since $a'$ and $b'$ are maximal in $T_w$, $c \in N$, and 
$T_z \restriction N = T_w$, we have that 
$$
c \le_{T_w} a', b'.
$$
Also note that since $F_z(a) \ne F_z(b)$, also 
$F_w(a') \ne F_w(b')$, which implies that $a' \ne b'$.
Therefore $c$ cannot equal $a'$ or $b'$, since $a'$ and $b'$ are incomparable. 
So $c <_{T_w} a', b'$. 
Since $F_w(a'), F_w(b') \in M \cap \beta_{M,N}$, it follows that 
$c \in M \cap \beta_{M,N}$, by requirement (4) of $w$ being a condition. 
So $c \in M$, and we are done.

\bigskip

\noindent Case B: $a$ and $b$ are in $T_s \setminus N$, and at least one of 
$F_z(a)$ or $F_z(b)$ is not in $N$. 
Without loss of generality, assume that $F_z(b) \notin N$. 
Since $F_z(b) \in M$, it follows that $M$ is not in $Sk(N)$, and hence $M$ 
is in $A_s$. 
Fix $i$ such that $M = M_i$.

Fix $x$ and $y$ minimal in $T_s \setminus N$ which are below $a$ and $b$ 
respectively. 
Note that as $c < N \cap \omega_1$, we have that 
$c <_{T_z} x, y$. 
If $x = y$, then $x <_{T_s} a, b$. 
It follows that $x \in M$, since $s$ is a condition. 
Since $c < x$, this implies that $c \in M$, and we are done.

So assume that $x \ne y$. 
Then $g(x) \ne g(y)$, since $g$ is injective. 
As $c <_{T_z} x, y$, and $g(x)$ and $g(y)$ are the immediate predecessors 
of $x$ and $y$ in $T_z$, we have that 
$c \le_{T_z} g(x), g(y)$. 
So $c \le_{T_w} g(x), g(y)$.

We claim that $c$ is below $\sigma(x)$ and $\sigma(y)$ in $T_w$. 
Note that $c$ and $\sigma(x)$ are comparable in $T_w$. 
For in case 1 of the definition of $g$, $g(x) = x^* <_{T_w} \sigma(x)$, 
and in case 2, $\sigma(x) \le_{T_w} g(x)$; both of these cases imply that 
$c$ and $\sigma(x)$ are comparable in $T_w$. 
Similarly, $c$ and $\sigma(y)$ are comparable in $T_w$.

But $x$ and $y$ are incomparable in $T_s$. 
So $\sigma(x)$ and $\sigma(y)$ are incomparable in $T_v$, since 
$\sigma$ is an isomorphism, and hence are incomparable in $T_w$. 
This implies that 
$$
c <_{T_w} \sigma(x), \sigma(y),
$$
since any other 
relation of $c$ with $\sigma(x)$ and $\sigma(y)$ would yield that 
$\sigma(x)$ and $\sigma(y)$ are comparable in $T_w$.

Now $\sigma(x) \le_{T_v} \sigma(a)$ and $\sigma(y) \le_{T_v} \sigma(b)$, 
since $\sigma$ is an isomorphism. 
As $T_w$ end-extends $T_v$, 
$\sigma(x) \le_{T_w} \sigma(a)$ and $\sigma(y) \le_{T_w} \sigma(b)$. 
But $c <_{T_w} \sigma(x), \sigma(y)$, as just noted. 
Therefore 
$$
c <_{T_w} \sigma(a), \sigma(b).
$$
We claim that $F_w(\sigma(a))$ and $F_w(\sigma(b))$ are in $L_i$. 
As $w$ is a condition, 
this implies that $c$ is in $L_i \cap \omega_1 = M \cap \omega_1$, 
which finishes the proof.

By our  assumption, 
$$
F_s(b) \in M_i \setminus N.
$$
By (5) of the main claim, 
$$
F_v(\sigma(b)) \in L_i.
$$
For $a$, there are two possibilities. 
If $F_s(a) \notin N$, then 
$$
F_s(a) \in M_i \setminus N,
$$
which by (5) of the main claim implies that 
$$
F_v(\sigma(a)) \in L_i.
$$
Otherwise $F_s(a) \in N$, so $F_s(a) \in M \cap N \subseteq M \cap \beta_{M,N}$. 
But $M \cap \beta_{M,N} \subseteq L_i$, so $F_s(a) \in L_i$.

\bigskip

\noindent Case C: At least one 
of $a$ or $b$ is not in $T_s \setminus N$, and at least one of $F_z(a)$ or $F_z(b)$ 
is not in $N$. 
Without loss of generality, assume that $a$ is not in $T_s \setminus N$. 
Then $a$ is in $T_w$. 
It follows that $F_z(a) = F_w(a)$, which is in $N$. 
Therefore $F_z(b) \notin N$. 
In particular, $b$ is in $T_s \setminus N$.
Also since $F_z(b) \in M \setminus N$, $M$ is not in $Sk(N)$. 
So $M$ is in $A_s$. 
To summarize, $a$ is in $T_w$, $b$ is in $T_s \setminus N$, 
$M$ is in $A_s$, and $F_z(b) \notin N$.

We have that 
$$
F_z(a) \in M \cap N \subseteq M \cap \beta_{M,N}.
$$
Since $F_z(b) \in M \setminus N$, $M \setminus \beta_{M,N}$ is nonempty. 
Fix $i = 1, \ldots, k$ such that $M = M_i$. 
Let $y$ be the minimal node of $T_s \setminus N$ below $b$.

\bigskip

\noindent Subcase C(i): There is a maximal node $d$ in $T_s$ above $y$ 
such that $F_s(d) \in N$. 
Note that $d \ne b$, since 
$F_s(b) = F_z(b) \notin N$. 
By the definition of $g$, 
we have that $g(y) = \sigma^+(d)$ and 
$F_w(g(y)) = F_s(d)$.

We claim that $c \le_{T_w} g(y)$. 
Since $c <_{T_z} b$, $y <_{T_z} b$, and $c < y$, it follows that 
$c <_{T_z} y$. 
Since $g(y)$ is the immediate predecessor of $y$ in $T_z$, 
we have that $c \le_{T_z} g(y)$. 
But $T_z$ end-extends $T_w$, so $c \le_{T_w} g(y)$.

So we have that $c \le_{T_w} a, g(y)$. 
Since $y <_{T_s} d$ and $\sigma$ is an isomorphism, 
$\sigma(y) <_{T_v} \sigma(d)$. 
So
$$
\sigma(y) <_{T_w} \sigma(d) \le_{T_w} \sigma^+(d) = g(y).
$$
As $c$ and $\sigma(y)$ are both below $g(y)$ in $T_w$, they are 
comparable in $T_w$.

We claim that $c <_{T_w} \sigma(y)$. 
Suppose for a contradiction that $\sigma(y) \le_{T_w} c$. 
Since $c <_{T_w} a$, it follows that 
$\sigma(y) <_{T_w} a$. 
Now $y <_{T_s} b$ implies that 
$\sigma(y) <_{T_v} \sigma(b)$, and hence 
$\sigma(y) <_{T_w} \sigma(b)$. 
Since $w \le v$, we can fix a maximal node $\sigma^+(b)$ of $T_w$ 
which is above $\sigma(b)$ such that 
$F_w(\sigma^+(b)) = F_v(\sigma(b))$. 
Then $\sigma(y) <_{T_w} \sigma^+(b)$.

Recall that $M = M_i$ and $L_i$ end-extends $M \cap \beta_{M,N}$. 
By (5) of the main claim, since $F_s(b) \in M \setminus N$, we have that 
$$
F_w(\sigma^+(b)) = F_v(\sigma(b)) \in L_i.
$$
Also as observed at the beginning of case C, 
$$
F_w(a) = F_z(a) \in M \cap \beta_{M,N} \subseteq L_i.
$$
Since $\sigma(y) <_{T_w} a, \sigma^+(b)$, by requirement (4) of 
$w$ being a condition it follows that 
$$
\sigma(y) \in L_i \cap \omega_1.
$$
But $L_i$ end-extends $M \cap \beta_{M,N}$ and 
$\omega_1 < \beta_{M,N}$. 
Therefore 
$$
\sigma(y) \in L_i \cap \omega_1 = M_i \cap \omega_1.
$$
But this contradicts (2) of the main claim.

This contradiction completes the proof that $c <_{T_w} \sigma(y)$. 
It follows that 
$$
c <_{T_w} \sigma(y) <_{T_w} \sigma(b) \le_{T_w} \sigma^+(b),
$$
so $c <_{T_w} \sigma^+(b)$. 
Also we are assuming that $c <_{T_w} a$. 
Now 
$$
F_w(a) = F_z(a) \in M \cap \beta_{M,N} \subseteq L_i,
$$ 
and by (5) of the main claim, 
$$
F_w(\sigma^+(b)) = F_v(\sigma(b)) \in L_i.
$$
Since $c <_{T_w} a, \sigma(y)$, by requirement (4) of $w$ being a condition, 
we have that 
$$
c \in L_i \cap \omega_1 = M \cap \omega_1.
$$
This completes the proof that $c$ is in $M$.

\bigskip

\noindent Subcase C(ii): There is no maximal node $d$ of $T_s$ above $y$ 
such that $F_s(d) \in N$. 
Then by the definition of $g$, $g(y) = y^*$, where $y^*$ is the 
predecessor of $y$ in $T_s$. 
Now $c$ is below $b$ in $T_z$ and hence below $y$. 
Since $g(y) = y^*$ is the immediate predecessor of $y$ in $T_z$, 
$c \le_{T_z} g(y)$. 
Therefore $c \le_{T_w} g(y)$. 
Hence 
$$
c \le_{T_w} g(y) = y^* = \sigma(y^*) <_{T_w} \sigma(y) \le_{T_w} 
\sigma(b) \le_{T_w} \sigma^+(b),
$$
where $\sigma^+(b)$ is the maximal node of $T_w$ above $\sigma(b)$ 
such that $F_v(\sigma(b)) = F_w(\sigma^+(b))$. 
So 
$$
c <_{T_w} a, \sigma^+(b).
$$

By property (5) of the main claim, since $F_s(b) \in M \setminus N$, 
$$
F_w(\sigma^+(b)) = F_v(\sigma(b)) \in L_i \setminus M.
$$
But $F_w(a) \in M$. 
It follows that $a \ne \sigma^+(b)$. 
Since $F_w(a) = F_z(a) \in M \cap \beta_{M,N} \subseteq L_i$ and 
$F_w(\sigma^+(b)) \in L_i$, 
by property (4) in the definition of $\p$ we have that 
$$
c \in L_i \cap \omega_1 = M \cap \omega_1.
$$
So $c \in M$, and we are done.
\end{proof}

\begin{proposition}
The forcing poset $\p$ is $\omega_2$-c.c.
\end{proposition}

\begin{proof}
We will use Lemma 4.3. 
Let $\theta > \omega_2$ be regular. 
Fix $N^* \prec H(\theta)$ of size $\omega_1$ such that 
$\p, \pi, \mathcal X \in N^*$ and 
$\beta^* := N^* \cap \omega_2 \in \Gamma$. 
Note that since $\Gamma$ is stationary, there are stationarily many such 
models $N^*$ in $P_{\omega_2}(H(\theta))$.

Observe that as $\pi \in N^*$ and $\pi : \omega_2 \to H(\omega_2)$ 
is a bijection, by elementarity we have that 
$$
N^* \cap H(\omega_2) = \pi[N^* \cap \omega_2] = \pi[\beta^*] 
= Sk(\beta^*),
$$
where the last equality holds by Lemma 1.3 and the fact that 
$\beta^* \in \Gamma$ implies that $Sk(\beta^*) \cap \omega_2 = \beta^*$. 
In particular, $N^* \cap \p \subseteq Sk(\beta^*)$.

We will prove that the empty condition is strongly $(N^*,\p)$-generic. 
By Lemma 4.3, this implies that $\p$ is $\omega_2$-c.c. 
So fix a set $D$ which is a dense subset of $N^* \cap \p$, and we will 
show that $D$ is predense in $\p$.

\bigskip

Let $q$ be a condition. 
We will find a condition in $D$ which is 
compatible with $q$. 
First, we extend $q$ to prepare for intersecting with $N^*$. 
Define $r$ by letting $T_r := T_q$, $F_r := F_q$, and 
$$
A_r := A_q \cup \{ M \cap \beta^* : M \in A_q \}.
$$
By the comments after Definition 7.5, $r$ is a condition, and clearly $r \le q$. 

We will show that there is a condition in $D$ which is compatible with $r$. 
Since $r \le q$, it follows that there is a condition in $D$ which is compatible 
with $q$, which completes the proof.

\bigskip

Note that since $\omega_1$ is a subset of $N^*$, the tree $T_r$ is actually 
a member of $N^*$. 

Let $M_1, \ldots, M_k$ list the elements $M$ of $A_r$ such that 
$M \setminus \beta^*$ is nonempty. 
Define $P_1, \ldots, P_k$ which are subsets of $|T_r|$ by letting 
$a \in P_i$ iff $a$ is maximal in $T_r$ and 
$F_r(a) \in M_i \setminus \beta^*$. 
Let $S$ be the set of maximal nodes $a$ of $T_r$ such that $F_r(a) < \beta^*$. 
For each $a \in S$, let $\xi_a := F_r(a)$.

To find a condition in $D$ which is compatible with $r$, we first need to find a 
condition in $N^*$ which reflects some information about $r$.

\bigskip

\noindent \textbf{Main Claim:} There exists a condition $v \in N^*$ satisfying:
\begin{enumerate}
\item $T_v = T_r$;
\item if $a$ if maximal in $T_r$ and $F_r(a) < \beta^*$, then 
$F_v(a) = F_r(a)$;
\item there are $L_1, \ldots, L_k$ in $A_v$ such that $L_i$ end-extends 
$M_i \cap \beta^*$ for all $i = 1, \ldots, k$;
\item if $a$ is maximal in $T_r$ and $F_r(a) \in M_i \setminus \beta^*$, 
then $F_v(a) \in L_i \setminus (M_i \cap \beta^*)$;
\item $A_r \cap P(\beta^*) \subseteq A_v$.
\end{enumerate}

\bigskip

We prove the claim. 
The objects $r$ and $M_1, \ldots, M_k$ witness the statement that 
there exists a condition $v$ and $L_1, \ldots, L_k$ satisfying:
\begin{enumerate}
\item[(i)] $T_v = T_r$;
\item[(ii)] if $a \in S$, then $F_v(a) = \xi_a$;
\item[(iii)] there are $L_1, \ldots, L_k$ in $A_v$ which end-extend 
$M_1 \cap \beta^*, \ldots, M_k \cap \beta^*$;
\item[(iv)] for all $a \in |T_v|$ and $i = 1, \ldots, k$, $a \in P_i$ iff 
$a$ is maximal in $T_v$ and $F_v(a) \in L_i \setminus (M_i \cap \beta^*)$;
\item[(v)] $A_r \cap P(\beta^*) \subseteq A_v$.
\end{enumerate}
Now the parameters which appear in the above statement, namely, 
$T_r$, $S$, $\langle \xi_a : a \in S \rangle$, 
$M_1 \cap \beta^*, \ldots, M_k \cap \beta^*$, 
$P_1, \ldots, P_k$, and $A_r \cap P(\beta^*)$, are all members of $N^*$. 
By the elementarity of $N^*$, we can fix a condition $v$ and 
$L_1, \ldots, L_k$ which are members of $N^*$ 
and satisfy the same statement. 
It is easy to check that $v$ satisfies the properties listed in the main claim.

\bigskip

Since $D$ is dense in $N^* \cap \p$, we can fix $w \le v$ in $D$. 
We will show that $r$ and $w$ are compatible, 
which finishes the proof.

\bigskip

We will define a condition $z = (T_z,F_z,A_z)$, and then show that 
$z \le w, r$. 
Let $A_z := A_r \cup A_w$.

Note that $T_w$ is an end-extension of $T_v = T_r$. 
Let us describe how to extend $T_w$ to $T_z$. 
In addition to having the original nodes of $T_w$, we will also split above certain 
nodes of $T_w$ as follows.

Let $Z$ be the set of maximal nodes $a$ of $T_r$ such that 
$F_r(a) \ge \beta^*$. 
For each $a \in Z$, let $a^+$ be the unique maximal node above $a$ in $T_w$ 
such that $F_v(a) = F_w(a^+)$. 
Now add above $a^+$ two immediate 
successors $a_0$ and $a_1$. 
This describes the tree $T_z$.

Define $F_z$ as follows. 
Let $b$ be a maximal node of $T_z$. 
Then either $b$ is equal to $a_0$ or $a_1$ for some $a \in Z$, or 
$b$ is maximal in $T_w$. 
In the second case, let $F_z(b) := F_w(b)$. 
In the first case, we let 
$$
F_z(a_0) := F_w(a^+) \ \textrm{and} \ F_z(a_1) := F_r(a).
$$
Note that $F_z(a_0) < \beta^*$ and $F_z(a_1) \ge \beta^*$.

\bigskip

This completes the definition of $z$. 
Let us prove that $z$ is a condition. 
Requirements (1) and (3) in the definition of $\p$ are immediate, using 
Proposition 3.11. 
For (2), the proof that $F_z$ is injective 
splits into a large number of cases, each of which is completely trivial. 
So we leave the straightforward verification to the reader. 
It remains to prove (4).

\bigskip

(4) Suppose that $M \in A_z$, 
and $c$ and $d$ are distinct maximal nodes of $T_z$ such that  
$F_z(c)$ and $F_z(d)$ are in $M$. 
Let $e <_{T_z} c, d$. 
We will show that $e \in M$.

\bigskip

\noindent Case 1: First assume that $F_z(c), F_z(d) < \beta^*$. 
Then $c$ is either maximal in $T_w$ or is equal to $a_0$ for some $a \in Z$, 
and similarly with $d$. 
It is easy to check that in each of these four cases, the node 
$e$ is below two maximal nodes of $T_w$ which $F_w$ maps into $M \cap \beta^*$. 
Since $M \cap \beta^* \in A_w$ and $w$ is a condition, it follows that 
$e \in M \cap \beta^*$. 
Hence $e \in M$.

\bigskip

\noindent Case 2: Now assume that $F_z(c), F_z(d) \ge \beta^*$. 
Then $c = a_1$ and $d = b_1$, where $a$ and $b$ are distinct nodes in $Z$. 
Since $e$ is below $c$ and $d$, $e$ is comparable with both $a$ and $b$. 
As $a$ and $b$ are incomparable in $T_v$ and hence in $T_w$, 
we cannot have that $a$ or $b$ is below $e$, 
since that would imply that $a$ and $b$ are comparable. 
Hence 
$$
e <_{T_w} a, b.
$$

Since $F_z(c) \in M \setminus \beta^*$, we can fix $i$ such that $M = M_i$. 
Then 
$$
F_r(a) = F_z(a_1) = F_z(c) \in M_i \setminus \beta^*
$$
and 
$$
F_r(b) = F_z(b_1) = F_z(d) \in M_i \setminus \beta^*.
$$
By (4) of the main claim, 
$$
F_z(a_0) = F_w(a^+) = F_v(a) \in L_i
$$
and 
$$
F_z(b_0) = F_w(b^+) = F_v(b) \in L_i.
$$
As $e$ is below $a$ and $b$, obviously $e <_{T_z} a_0, b_0$. 
By Case 1, $e \in L_i \cap \omega_1 \subseteq M$.

\bigskip

\noindent Case 3: Assume that 
$F_z(c) \ge \beta^*$ and $F_z(d) < \beta^*$. 
Then $c = a_1$ for some $a \in Z$, and 
$d$ is either equal to $b_0$ for some $b \in Z$ or is maximal in $T_w$. 
Then 
$$
F_z(c) = F_z(a_1) = F_r(a).
$$

Since $F_z(c) \in M \setminus \beta^*$, we can fix $i$ such that $M_i = M$. 
Then 
$$
F_r(a) = F_z(c) \in M_i \setminus \beta^*.
$$
By (4) of the main claim, 
$$
F_z(a_0) = F_w(a^+) = F_v(a) \in L_i \setminus (M \cap \beta^*).
$$
Note that $d$ is not equal to $a_0$. 
For otherwise $F_z(d) \in L_i \setminus (M \cap \beta^*)$, which contradicts 
our assumption that $F_z(d) \in M \cap \beta^*$.

Now $e <_{T_z} c = a_1$ implies that $e \le_{T_w} a^+$. 
But if $e = a^+$, then $a^+ <_{T_z} d$, which implies that 
$d = a_0$, which we just showed is not true. 
So $e <_{T_w} a^+$.
As observed above, $F_w(a^+) \in L_i$.

If $d$ is maximal in $T_w$ and not equal to any $b_0$, 
then 
$$
F_w(d) = F_z(d) \in M \cap \beta^* \subseteq L_i.
$$
Since $e <_{T_w} a^+, d$, and 
$F_w(a^+)$ and $F_w(d)$ are in $L_i$, then since $w$ is a condition, 
$$
e \in L_i \cap \omega_1 \subseteq M.
$$
So $e \in M$, and we are done.

The other possibility is that 
$d$ is not maximal in $T_w$, and $d = b_0$ for some $b \in Z$. 
We observed above that $d \ne a_0$. 
Therefore $a \ne b$. 
So $a^+ \ne b^+$. 
Since $e$ is below $a_0$ and $b_0$, we have that 
$e \le_{T_w} a^+, b^+$. 
Since $a^+$ and $b^+$ are distinct maximal nodes of $T_w$, 
they are incomparable, and hence 
$e <_{T_w} a^+, b^+$. 
But $F_w(a^+) \in L_i$, \
and 
$$
F_w(b^+) = F_z(b_0) = F_z(d) \in M \cap \beta^* \subseteq L_i.
$$
Since $w$ is a condition, it follows that 
$$
e \in L_i \cap \omega_1 \subseteq M.
$$
So $e \in M$, and we are done.

\bigskip

\noindent Case 4: 
The case when $F_z(d) \ge \beta^*$ and $F_z(c) < \beta^*$ is the same 
as case 3, with the roles of $c$ and $d$ reversed.

\bigskip

This completes the proof that $z$ is a condition. 
Now we show that $z \le w, r$. 
Obviously $T_z$ end-extends $T_w$ and $T_r$, 
and by definition, $A_r$ and $A_w$ are subsets of $A_z$.

To show that $z \le w$, let $c$ be maximal in $T_w$. 
If $c$ remains maximal in $T_z$, then $F_z(c) = F_w(c)$, and we are done. 
Otherwise $c = a^+$ for some $a \in Z$, and $a_0$ and $a_1$ were added above $c$. 
By definition, 
$$
F_z(a_0) = F_w(a^+) = F_w(c).
$$
This proves that $z \le w$.

To show that $z \le r$, 
suppose that $d$ is maximal in $T_r$. 
There are two cases depending on whether $F_r(d) < \beta^*$ or 
$F_r(d) \ge \beta^*$. 
Assume first that $F_r(d) < \beta^*$. 
Then by (2) of the main claim, $F_r(d) = F_v(d)$. 
Let $d^+$ be the unique maximal node of $T_w$ above $d$ such that 
$F_w(d^+) = F_v(d)$. 
Then by the definition of $T_z$, $d^+$ is still maximal in $T_z$, 
and 
$$
F_z(d^+) = F_w(d^+) = F_v(d) = F_r(d).
$$
Now assume the other case that $F_r(d) \ge \beta^*$. 
Then $d \in Z$, and by the definition of $T_z$ and $F_z$, $d_1$ is a maximal node 
of $T_z$ above $d$, and 
$$
F_z(d_1) = F_r(d).
$$
This proves that $z \le r$.
\end{proof}

\bibliographystyle{plain}
\bibliography{paper21}

\end{document}